\documentclass[reqno]{amsart}
\usepackage{amsmath}%
\usepackage{amsfonts}%
\usepackage{amssymb}%
\usepackage{lscape}%
\usepackage[matrix,arrow,curve]{xy}%
\usepackage{graphicx}
\usepackage{tabularx}

\numberwithin{equation}{section}

\theoremstyle{plain}

\newtheorem{theorem}{Theorem}

\newtheorem{proposition}{Proposition}
\newtheorem{lemma}{Lemma}[section]

\theoremstyle{definition}

\newtheorem{ProofDixEquiv}{Proof of Proposition \ref{dix.equiv}\hskip0pt}
                                                    
\newtheorem{remark}{Remark}

\def\Re{\text{\rm Re}\,}
\def\Im{\text{\rm Im}\,}

\def\refBott  {1}  
\def\refDixon {2}
\def\refGGLS  {3}
\def\refGKZ   {4}
\def\refGLS   {5}
\def\refKov   {6}
\def\refLev   {7}
\def\refMaTo  {8}  
\def\refWaHu  {9}  
\def\refWh    {10} 
\def\refWu    {11} 

\def\DixI  {D1}
\def\DixII {D2}
\def\SDixI {SD1}
\def\SDixII{SD2}
\def\PDixII{PD2}
\def\DA    {CDA}

\begin{document}


\author{M.\,D.~Kovalev and S.\,Yu.~Orevkov}
\address{Moscow State University (M.K.)}
\address{IMT, l'universit\'e Paul Sabatier, Toulouse; Steklov Math. Inst., Moscow (S.O.)}

\title{  Complete bipartite graphs flexible in the plane}

\markboth{M.\,D.~Kovalev and S.\,Yu.~Orevkov}
         {Complete bipartite graphs flexible in the plane}

\maketitle

\begin{abstract}
A complete bipartite graph $K_{3,3}$, considered as a planar linkage with joints
at the vertices and with rods as edges, in general admits only motions as a whole,
i.e., is inflexible. Two types of its paradoxical mobility were found by Dixon in 1899.
Later on, in a series of papers by different authors, the question of flexibility of
$K_{m,n}$ was solved for almost all pairs $(m,n)$.
In the present paper, we solve it for all complete bipartite graphs in the
Euclidean plane as well as in the sphere and in the hyperbolic plane.
We give independent self-contained proofs without extensive computations
which are almost the same in the Euclidean, hyperbolic and spherical cases.
\end{abstract}

\section{Introduction. Main results}
\label{sec1}

We find necessary and sufficient conditions (Theorem~\ref{t1} and
Remarks~\ref{z1} and \ref{z2}) for the flexibility of frameworks corresponding
to complete bipartite graphs $K_{m,n}$ in the Euclidean plane $E^2$.
In \S\ref{sect.hyp} and \S\ref{sect.sph} we solve the same problem for
the hyperbolic plane $H^2$ and for the sphere $S^2$.
Most of the results are not new but we give complete
self-contained proofs which are almost the same for $E^2$, $H^2$, and $S^2$.
We do not know whether the hyperbolic case might have any independent interest,
but it serves as a very convenient bridge between the proof in
the Euclidean and in the spherical cases. Namely, when passing from $E^2$ to $H^2$,
only some formulas are changed but the geometric and combinatorial arguments
are exactly the same.
Whereas, when passing from $H^2$ to $S^2$, all the formulas are
the same (just with $\cos$ and $\sin$ instead of $\cosh$ and $\sinh$), and
only the combinatorial part is somewhat extended.

We define a {\it planar framework corresponding to the complete bipartite graph}
$K_{m,n}$ ($(m,n)$-{\it framework} for short) as a collection
of points in the Euclidean plane
${\bf p}=(p_1,\dots,p_m;\,q_1,\dots,q_n)$ 
such that $p_i\ne q_j$ for all $i$, $j$. The {\it parts} (of the bipartite graph)
are $(p_1,\dots)$ and $(q_1,\dots)$. Speaking of $(m,n)$-frameworks,
we call the points $p_i$ and $q_j$ {\it joints} and pairs of points $(p_i,q_j)$
from different parts {\it rods}.
We say that an $(m,n)$-framework $\bf p$ is {\it non-overlapping} if
all its joints are pairwise distinct.
Finally, we say that an $(m,n)$-framework is {\it flexible} if it admits
a {\it flex}, that is
a continuous non-constant motion of its joints
${\bf p}(t)=\big(p_1(t),\dots,q_n(t)\big)$, $t\in[0,1]$, such that
${\bf p}(0)=\bf p$, the lengths of the rods are constant, i.~e.~%
$|p_i(t)-q_j(t)|$ does not depend on $t$ for each $i,j$,
and some two joints from different parts do not move:
$p_{i_0}(t)=p_{i_0}$ and $q_{i_0}(t)=q_{i_0}$.
These definitions evidently can be extended to all connected graphs but we do not
need it.%
\footnote{In the literature on mechanics, rigid frameworks are usually called
{\it trusses}, and flexible ones are called, depending on the context,
{\it mechanisms} or {\it states} of a mechanism.}

\begin{theorem}\label{t1}
Let $\min(m,n)\ge 3$. Then a non-overlapping
$(m,n)$-framework is flexible if and only if one of the following conditions holds.

 (\DixI) The points $p_1,\dots,p_m$ lie on a line $P$, the points $q_1,\dots,q_n$
 lie on another line $Q$, and these two lines are orthogonal to each other.

 (\DixII)
 One can choose an orthogonal coordinate system and two rectangles with sides
 parallel to the axes and with common center of symmetry at the origin so that 
 $p_1,\dots,p_m$ are at the vertices of one rectangle
 and  $q_1,\dots,q_n$ are at the vertices of the other one. Since all points are
 distinct, we have in this case $m\le 4$ and $n\le 4$.
\end{theorem}

\begin{remark}\label{z1}
It is easy to see that any $(1,n)$-framework is flexible
(and has \hbox{$n-1$} degrees of freedom),
and a non-overlapping $(2,n)$-framework is flexible if and only if it does not
contain a quadruple of joints $p_i,q_j,p_k,q_l$ placed in this order on some straight
line (cf.~Lemma~\ref{nuiz}). All non-overlapping flexible $(m,n)$-frameworks
with $\min(m,n)\ge 2$ have one degree of freedom.
\end{remark}

\begin{remark}\label{z2}
It is evident that an $(m,n)$-framework $\bf p$ with overlapping joints is
flexible if and only if so is the non-overlapping
bipartite framework $\overline{\bf p}$, obtained by identifying each pair of
overlapping joints. It is also clear that the number of degrees of freedom
$d({\bf p})$ of $\bf p$ is equal to
$\max_{\bf q} d(\overline{\bf q})$ where the maximum is taken over all
frameworks $\bf q$ obtained from $\bf p$ by small flexes.
Thus $d({\bf p})=d(\overline{\bf p})\le 1$ unless $p_1=\dots=p_m$ (in which case
$d({\bf p})=n-1$) or, symmetrically, $q_1=\dots=q_n$, $d({\bf p})=m-1$.
\end{remark}

In the case $m=n=3$, the flexible frameworks (\DixI) and (\DixII) were discovered 
by Dixon (see \cite[\S27(d), \S28(n)]{\refDixon}). They are called
{\it Dixon mechanisms of the first and second kind} respectively.
We shall use these names for any $m,n\ge 3$.
The Dixon mechanism of the second kind for $m=n=4$ apparently was first described
by Bottema in \cite{\refBott} (see also \cite{\refWu}).
One can equivalently reformulate (\DixI) and (\DixII) in terms of the rod lengths.
In the case of (\DixII) we do it for $(m,n)=(3,3)$ only,
but analogous conditions for
$(3,4)$ and $(4,4)$ can be easily derived.

\begin{proposition}\label{dix.equiv}
(a). A non-overlapping $(m,n)$-framework is a Dixon mechanism of the first kind
(see Fig.~\ref{fig.dix1}(a)) if and only if, for each cycle $p_iq_jp_kq_l$,
the sums of squared lengths of the opposite sides are equal. The number of these
conditions is $\binom{m}2\binom{n}2=\frac14(m^2-m)(n^2-n)$ but it is easily seen that only $(m-1)(n-1)$ of them are independent; one can choose, for example,
only the conditions corresponding to the cycles $p_iq_jp_kq_l$ with fixed
$i$ and $j$ (in particular, four conditions are independent among the nine ones
when $m=n=3$). 

\smallskip
(b). A flexible non-overlapping $(3,3)$-framework $(p_0,p_1,p_2;\,q_0,q_1,q_2)$
is a Dixon mechanism of the second kind if and only if, up to renumbering of
the vertices in the parts,
$|q_0p_0|= |q_1p_1|=|q_2p_2|=a$, $|q_0p_1|=|q_1p_0|=b$, $|q_0p_2|=|q_2p_0|=c$,
$|q_1p_2|=|q_2p_1|=d$ (see Fig.~\ref{fig.dix2})
and the relation $a^2+c^2=b^2+d^2$ holds.
In this case, all the 4-cycles twice including $a$ are parallelogrammatic,
i.e., have opposite sides of equal lengths.
\end{proposition}

\begin{remark}\label{z3}
The following example shows that Statement (b) of
Proposition~\ref{dix.equiv} is wrong without the flexibility assumption: \hskip-1pt
$p_0=(b,0)$, $p_1=(0,a)$, $p_2=(d,0)$,
$q_0=(b,a)$, $q_1=(0,0)$, $q_2=(d,a)$,
where $a,b,d$ are positive, $b\ne d$, and $bd=a^2$.
Indeed, all the conditions on the rod lengths are satisfied in this case, but
the framework is not a Dixon 2nd kind mechanism.
This is also an example of two non-overlapping $(3,3)$-frameworks with
equal lengths of the respective rods, one of whom is flexible
(a Dixon mechanism of the 2nd kind) and the other one is rigid by
Proposition~\ref{dix.equiv}. This example is a particular case of the example
in Fig.~\ref{fig.dix1}(b).
\end{remark}

\begin{figure}[ht]
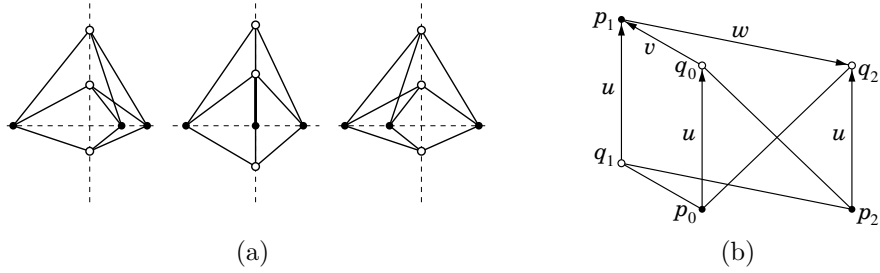

\centering
\begin{tabular}{ccc}
\includegraphics[width=67 mm]{dix1-def.eps}&\hbox to 5mm{}&
\includegraphics[width=38 mm]{dix-eq.eps}
\\
(a) && (b)
\end{tabular}
\caption{(a) Dixon mechanism of the first kind in motion.
(b) Rigid $(3,3)$-framework with lengths as in Dixon mechanism of the second kind.
Here the vectors $u$, $v$, and $w$ satisfy the relations $u^2+vw=u(v+w)=0$.}
\label{fig.dix1}
\end{figure}

\begin{figure}[ht]
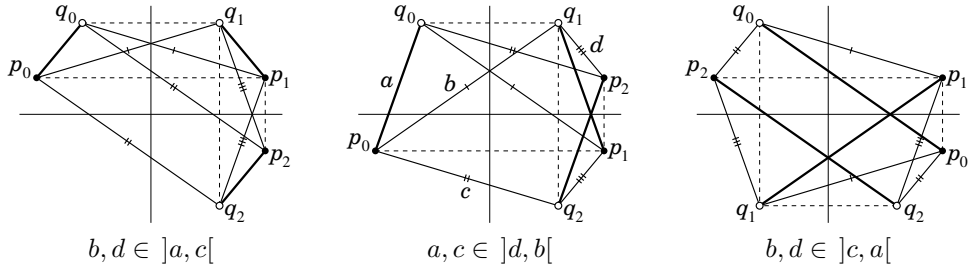

\begin{tabular}{ccccc}
\centering
\includegraphics[width=38 mm]{dix2a.eps}&\hskip 4mm&
\includegraphics[width=38 mm]{dix2b.eps}&\hskip 4mm&
\includegraphics[width=38 mm]{dix2c.eps}
\\
$b,d\in\,{]}a,c{[}$ && $a,c\in\,{]}d,b{[}$ && $b,d\in\,{]}c,a{[}$
\end{tabular}
\caption{Dixon mechanism of the second kind. The indicated conditions on the lengths
are determined (up to exchange of $b$ and $d$) by the fact,
whether $p_1$ and $q_1$ are in the same quadrant, in the adjacent ones,
or in the opposite ones. In all cases we see that $p_0q_0p_2q_2$ is a parallelogram
whereas $p_1q_1p_kq_k$, $k=0,2$, are antiparallelograms.}
\label{fig.dix2}
\end{figure}

The proof of Proposition~\ref{dix.equiv} is not difficult and it is given
at the end of this section. Notice that
Theorem~\ref{t1} is proven in \cite{\refMaTo} for $m\geq 3$ and $n\ge 5$. 
Also, as proven in \cite{\refWaHu}, the lengths of the rods of flexible
non-overlapping $(3,3)$-frameworks are as in Proposition~\ref{dix.equiv}.
This fact combined with Proposition~\ref{dix.equiv} yields
Theorem~\ref{t1} for $m=n=3$. Another proof of Theorem~\ref{t1} for $m=n=3$
is given in \cite[Example 4.3]{\refGLS}.
The reduction of the general case to the case $m=n=3$ is very simple.
It is as follows.

\begin{proof}[Proof of Theorem \ref{t1}
under the assumption that it holds for \hbox{$m=n=3$}.]
Consider a non-over\-lapping $(m,n)$-framework with $n\ge m\ge 3$.
The points $p_1,p_2,p_3$ and $q_1,q_2,q_3$ satisfy one of the conditions
(\DixI) or (\DixII). 

Let they satisfy (\DixI).
Then $p_1,p_2,p_3$ and $q_1,q_2,q_j$, $j\ge 3$, do not satisfy (\DixII).
Hence, since the subgraph spanned by them is flexible, they satisfy
(\DixI), i.e., $q_j$ is on the line $Q$.
Thus $q_1,\dots,q_n$ are all on $Q$.
By the same reason, $p_1,\dots,p_m$ are all on $P$.

Now suppose that $p_1,p_2,p_3$ and $q_1,q_2,q_3$ satisfy (\DixII).
Consider the $(3,3)$-framework
$(p_1,p_2,p_3;\,q_1,q_2,q_j)$, $j\ge 3$. It also satisfies (\DixII)
because (\DixI) cannot hold (for $p_1,p_2,p_3$ are not collinear).
A priori (\DixII) could hold for another choice of the axes of symmetry,
however, the triangle $p_1p_2p_3$ has a single pair of mutually orthogonal sides,
which uniquely determines the rectangle, and hence, it determines the axes.
Notice also that a rectangle which is symmetric with respect to the origin and which has
sides parallel to the axes, is determined by any of its vertices.
Therefore $q_1,q_2,q_3$, and $q_j$ are at the vertices of the same rectangle.
Analogously, $p_1,\dots,p_m$ are at the vertices of the same rectangle.
The theorem is proven.
\end{proof}

The rest of this section is devoted to the proof of Proposition~\ref{dix.equiv}.
In \S\S\ref{main.idea}--\ref{sect.EOP} we give a self-contained proof of
Theorem~\ref{t1} for $m=n=3$. In \S\ref{sect.hyp} and \S\ref{sect.sph} we
treat the hyperbolic and spherical cases respectively.

\begin{lemma}\label{fix} Let $m,n\ge2$. Then any flex of a non-overlapping
$(m,n)$-framework (see the definition above) leaves unmovable two joints only.
\end{lemma}

\begin{proof} The statement follows from the fact that the immobility of any
two joints of one part implies the immobility of all joints of the other part.
\end{proof}

\begin{lemma}\label{4kd} The diagonals of a quadrilateral (maybe, self-crossing) are
orthogonal if and only if the sums of the squared lengths of the opposite sides
are equal.
\end{lemma}

\begin{proof} Let $u,v,w$ be the vectors of three consecutive sides of the quadrangle.
Then twice the dot product of the diagonals is
$2(u+v)(v+w)=v^2+(u+v+w)^2-u^2-w^2$.
\end{proof}

It is clear that any parallelogrammatic non-overlapping quadrangle is
either a {\it parallelogram} (when its opposite sides are parallel) or
an {\it antiparallelogram} (when its diagonals are parallel).
It is both simultaneously if and only if it is {\it degenerate},
i.e., all its verices are collinear.

\begin{ProofDixEquiv}
(a). The statement follows from Lemma~\ref{4kd}.

(b). The condition on the lengths is derived from (\DixII) by a direct computation.
Let us prove the inverse implication.
Since  $a^2+c^2=b^2+d^2$, Lemma~\ref{4kd} implies that the diagonals
of the quadrangle $p_0q_1p_1q_2$ are mutually orthogonal.
The same is true for the diagonals of $q_0p_1q_1p_2$
(see Fig.~\ref{fig.dix2}),
i.e., $p_0p_1\perp q_1q_2$ and $q_0q_1\perp p_1p_2$.
By hypothesis, the cycles $\Pi_{ij}=p_iq_ip_jq_j$, $i<j$, are parallelogrammatic.

Suppose that both $\Pi_{01}$ and $\Pi_{12}$ are non-degenerate parallelograms
(see Fig.~\ref{fig.dix1}(b)).
Then $p_0q_0q_2p_2$ also is a parallelogram and, since its both diagonals are
of length $c$, it is a rectangle with sides $a$ and $\sqrt{c^2-a^2}$.
Hence any flex fixing $p_0$ and $q_0$ fixes
$p_2$ and $q_2$ as well, which contradicts Lemma~\ref{fix}.

The obtained contradiction shows that $\Pi_{01}$ or $\Pi_{12}$ is an
antiparallelogram (maybe, degenerate). Let it be $\Pi_{01}$
(the case of $\Pi_{12}$ is analogous).
Then $q_1q_2{\perp}p_0p_1||q_0q_1{\perp}p_1p_2$, hence $q_1q_2\,||\,p_1p_2$,
i.e., $\Pi_{12}$ also is an antiparallelogram.
Hence $q_0$ and $q_2$ are symmetric to $q_1$ with respect to the mutually
orthogonal symmetry axes of these antiparallelograms (see Fig.~\ref{fig.dix2}).
The same is true for $p_0, p_2$, and $p_1$. The proposition is proven.
\end{ProofDixEquiv}

\medskip
\noindent
{\bf Acknowledgement.} We thank Matteo Gallet for informing us about the
papers \cite{\refGGLS} and \cite{\refGLS} and for some comments on them.


\section{A general scheme of the proof of Theorem~\ref{t1} for $m=n=3$.}\label{main.idea}

Consider a flex of a non-overlapping $(3,3)$-framework
${\bf p}=(p_0,p_1,p_2;\,q_0,q_1,q_2)$ such that
the joints $p_0$ and $q_0$ are fixed.
Then $p_1,p_2,q_1,q_2$ move along the circles which we denote by
$P_1,P_2,Q_1,Q_2$ respectively.
Forget for a while the joint $p_2$. Then generically (when the segment
$q_1p_1$ is not orthogonal to $P_1$) the displacement of
$q_1$ uniquely determines the displacement of $p_1$, which, in its turn,
generically determines the displacement of $q_2$. We obtain a dependence
$q_2={\mathcal F}_1(q_1)$ (see Fig.~\ref{fig.transm}).%
\footnote{%
In engineering, 
this dependence is called {\it
zero order transmission function} or {\it position function}
(see, e.g., \cite[\S 41]{\refLev}).}
Analogously, $p_2$ ensures a dependence $q_2={\mathcal F}_2(q_1)$.
In order for our $(3,3)$-framework not to be jammed, the functions
${\mathcal F}_1$ and ${\mathcal F}_2$ should coincide.
The point $(q_1,{\mathcal F}_i(q_1))$, $i=1,2$, moves along a certain
real algebraic curve $C_i$ on the torus
$Q_1\times Q_2$. The flexibility of $\bf p$ requires that $C_1$ and $C_2$
have an irreducible component in common.

\begin{figure}[ht]
\centering
\includegraphics[width=100 mm]{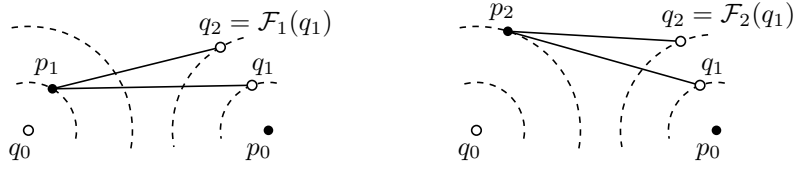}%
  \put(-270,35){$p_1$}\put(-188,35){$q_1$}\put(-18,35){$q_1$}
  \put(-208,50){$q_2={\mathcal F}_1(q_1)$}\put(-98,56){$p_2$}
  \put(-33,53){$q_2={\mathcal F}_2(q_1)$}
  \put(-280,3){$q_0$}\put(-190,3){$p_0$}
  \put(-110,3){$q_0$}\put(-20,3){$p_0$}
\caption{The transmission functions ${\mathcal F}_1$ and ${\mathcal F}_2$}.
\label{fig.transm}
\end{figure}

Let us proceed to a more formal exposition. Fix two points
$p_0,q_0\in{\mathbb{R}}^2$. Without loss of generality we may set
$q_0=(0,0)$ and $p_0=(r,0)$. Fix real positive numbers
$r_{ij}$, $i,j\in\{0,1,2\}$, $r_{00}=r$.
Denote also $R_i=r_{i0}$ and $r_i=r_{0i}$, $i=1,2$.
Let $M$ be the set of all quadruples $(p_1,p_2;\,q_1,q_2)$ such that
$|p_iq_j|=r_{ij}$, $i,j\in\{0,1,2\}$. It is natural to consider $M$ as the
moduli space of $(3,3)$-frameworks with a given matrix of the lengths.
Abusing the language, we shall call the elements of $M$ also 
$(3,3)$-frameworks implicitly assuming them to include $p_0$ and $q_0$.
As above, we define the circles
$$
   P_k=\{p_k\in{\mathbb{R}}^2 \;:\; |q_0p_k|=R_k\},\qquad
   Q_k=\{q_k\in{\mathbb{R}}^2 \;:\; |p_0q_k|=r_k\},\qquad k=1,2,
$$
and set $Q=Q_1\times Q_2$.
For $k=1,2$, consider the space of $(2,3)$-frameworks
$(p_0,p_k;\,q_0,q_1,q_2)$ with these lengths:
$$
   M_k=\{(p_k,q_1,q_2)\in P_k\times Q\;:\; |p_kq_1|=r_{k1},\; |p_kq_2|=r_{k2}\}.
$$
Set $C_k=\tau_k(M_k)$ where $\tau_k:P_k\times Q\to Q$, $k=1,2$, are the standard
projections (these are the curves appeared in the above discussion
of transmission functions).
It is clear that generically $C_1$ and $C_2$ are algebraic curves on $Q$
(though if, for example, $M$ has an element such that $p_0=p_k$, then $C_k=Q$).
Let us find the defining equations for $C_1$ and $C_2$.
As in \cite{\refWaHu}, we parametrize the circle $Q_j$, $j=1,2$, by
a complex number $t_j$, running over the circle $|t_j|=r_j/r$ in the complex plane.
The coordinates of the vector $p_0q_j$ are $(\Re t_j,\Im t_j)$,
in other words, the parameter of $q_j$ is the image of the vector $p_0q_j$ under
the standard identification of $\mathbb{R}^2$ with ${\mathbb{C}}$.
Analogously we choose parameters $T_i$ on the circles $P_i$. 
In these coordinates, the conditions $|p_iq_j|=r_{ij}$, $i,j=1,2$, 
take the form $f_{ij}(T_i,t_j)=0$ where $f_{ij}$ is the numerator of the rational
function obtained from the expression
$r^2(1+T_i - t_j)(1+\overline T_i - \overline t_j)-r_{ij}^2$
by the replacement $\overline T_i = R_i^2/(r^2 T_i)$,
$\,\overline t_j = r_j^2/(r^2 t_j)$,
i.e., (cf.~\cite[eqs.~(6)--(9)]{\refWaHu})
$$
   f_{ij}(T_i,t_j) = (1 + T_i - t_j)(r^2 T_it_j + R_i^2 t_j - r_j^2 T_i)
                        - r_{ij}^2T_it_j.
$$
$C_i$ is the projection of set of solutions of the system of equations
$f_{i1}=f_{i2}=0$, hence
it is given by the equation $F_i(t_1,t_2)=0$ where
\begin{equation}\label{F[i]}
      F_i(t_1,t_2)=R_i^{-2}\,\text{Res\,}_{T_i}(f_{i1},f_{i2})
\end{equation}
(see Remark~\ref{rem.GKZ} below).
The expression for $F_i$ (as a polynomial in $t_1$, $t_2$ and in all the $r_{ij}$'s)
has 126 monomials and we have $\deg_{t_j}F_i=4$.
In the case of a non-overlapping flex, the images of $M$ in $Q_j$
are not discrete by Lemma~\ref{fix}, which implies the following fact.

\begin{lemma}\label{lem.res}
If $M$ contains a flexible non-overlapping framework, then
\begin{equation}\label{eq.res}
    \text{\rm Res\,}_{t_1}(F_1,F_2)=\text{\rm Res\,}_{t_2}(F_1,F_2)=0.
\end{equation}
\end{lemma}

Thus the search of all flexible $(3,3)$-frameworks is reduced to a computation
of the resultant of $F_1$ and $F_2$ and a solution of the system of equations
obtained by equating all its coefficients to zero.
This is the way Walter and Husty have obtained in \cite{\refWaHu} the result
(mentioned in the introduction) that the lengths of the rods of flexible
non-overlapping $(3,3)$-frameworks are always as in Dixon's mechanisms.
According to \cite{\refWaHu}, $\text{Res}(F_1,F_2)$
has 4.900.722 monomials, and it is said in \cite{\refWaHu} that
``the computations are very extensive with respect to time and memory''.
Also, as far as we understood from \cite{\refWaHu}, one needs to do some
programming to interpret the solutions obtained with Maple or Singular.

When we started working on flexible $(3,3)$-frameworks (not knowing
about the paper \cite{\refWaHu}), we also tried to solve this system of
equations. However, we did not succeed to overcome the computational difficulties
and looked for how to avoid them.%
\footnote{%
    Probably, we would not do it,
    if we were acquainted that time with the paper \cite{\refWaHu}.}
So we found the proof exposed below. The longest computation in our proof is
that of the resultant (\ref{res.bcd}), which takes 25\,ms of CPU time.
It should be pointed out that the choice of the parameters $T_i$ and $t_j$ borrowed
from \cite{\refWaHu} further simplified the computations in Lemma~\ref{lem.propor}
(initially, we used the standard parametrization of the circle by the
tangent half-angle).

The outline of our proof is as follows.
If $M$ contains a flexible non-overlapping $(3,3)$-framework, then the curves
$C_1$ and $C_2$ have a common component, i.e., the polynomials $F_1$ and $F_2$
have a common divisor.
If one of $F_1$, $F_2$ is irreducible, they are proportional.
This gives equations, which are easy to solve.

If $F_1$ and $F_2$ have a common divisor not being proportional,
we look how the complexifications of the curves $M_i$, $C_i$,
$C_{ij}=\{f_{ij}=0\}$, and $P_i$ are mapped to each other under the projections.
A not difficult study shows that, for each $i=1,2$,
either one of $C_{ij}$ is reducible, or
the projections $C_{i1}\to P_i$ and $C_{i2}\to P_i$ are ramified over the same points.
Both conditions lead to equations which allow us to conclude that the framework
contains either a parallelogrammatic cycle or a {\it deltoid}
(a 4-cycle symmetric with respect to a diagonal) arranged in a certain way
with respect to $p_0$ and $q_0$. Varying the choice of the fixed joints we arrive
either to Dixon-1 or to a framework which contains
three parallelogrammatic cycles adjacent to each other as in Dixon-2.
In the latter case, the resultant of $F_1$ and $F_2$ is easy to compute.

\begin{remark}\label{rem.GKZ}
Even when the coefficients of $T_i^2$ in $f_{ij}$ vanish, the resultant
in (\ref{F[i]}) is understood as the resultant of quadratic polynomials
($R_{2,2}$ in the notation of \cite[Ch.~12]{\refGKZ}, that is the determinant of
the $4\times 4$ Sylvester matrix). Similarly, the resultants in (\ref{eq.res}) and
the discriminants $D_j$ and $\Delta_j^\pm$ in \S\ref{sect.reduc} always correspond to
$R_{4,4}$ and $D_2$ from \cite[Ch.~12]{\refGKZ}.
\end{remark}


\section{Preliminary lemmas}\label{sect.prelim}

\begin{lemma}\label{nuiz}
{\rm(Immediate from Lemma~\ref{fix}.)}
If an $(m,n)$-framework, $m,n\geq 2$, contains a \hbox{4-cycle} with a rod
whose length is equal to the sum of the lengths of the three other rods of the cycle,
then the framework is not flexible.
\end{lemma}

\begin{lemma}\label{collin}
Let ${\bf p}=(p_0,p_1,p_2;\,q_0,q_1,q_2)$ be a flexible non-overlapping
$(3,3)$-framework. Suppose that
$|p_0q_j|=|p_1q_j|$ for all $j=0,1,2$, i.e.,
the joints $p_0$ and $p_1$ are equidistant from $q_0,q_1,q_2$.
Then ${\bf p}$ is a Dixon mechanism of the first kind.
\end{lemma}

Since flexible frameworks are infinitesimally flexible, this
lemma follows from Whiteley Theorem%
\footnote{It was essentially used in \cite{\refMaTo} in the proof of Theorem~\ref{t1} for
$m\ge 3$ and $n\ge 5$.}
\cite{\refWh} according to which
{\sl a non-overlapping $(m,n)$-framework with $\min(m,n)\ge 3$ is infinitesimally
flexible if and only if either all joints lie on a second order curve, or
all joints of one part and at least one joint of the other part are collinear}
(for $m=n=3$, the second condition is a particular case of the first one).
However, since we are giving a self-contained proof of Theorem~\ref{t1}, let us prove
Lemma~\ref{collin} directly.

\begin{proof}
Denote the rod lengths by $r_i=|p_0q_i|=|p_1q_i|$, $R_i=|p_2q_i|$,
$i=0,1,2$.
Consider a continuous deformation ${\bf p}(t)$.
The equidistance condition implies that the points $q_j$ rest collinear
and $q_0q_1\perp p_0p_1$ during the deformation.
Hence, without loss of generality, we may assume that the $q_j$'s remain
on the axis $y=0$, whereas $p_0$ and $p_1$ remain on the axis $x=0$.
Set $q_i=(x_i,0)$, $i=0,1,2$,
and denote the $x$-coordinate of $p_2$ by $a$. Then
\begin{equation}\label{eq.collin}
   r_0^2 - x_0^2 = r_j^2 - x_j^2, \qquad
   R_0^2 - (x_0-a)^2 = R_j^2 - (x_j-a)^2, \qquad j=1,2.
\end{equation}
Differentiating these identities with respect to $t$, we obtain a system of four
linear homogeneous equations for $x'_0,x'_1,x'_2,a'$.
The determinant is $a(x_0-x_1)(x_0-x_2)(x_1-x_2)$.
The flexibility implies the existence of a non-zero solution, thus $a=0$.
\end{proof}


\section{General case: $F_1$ and $F_2$ are proportional}\label{sect.gen}

Let the notation be as in \S\ref{main.idea}. Suppose that
$M$ contains a flexible non-overlapping $(3,3)$-framework
${\bf p}=(p_0,p_1,p_2;\,q_0,q_1,q_2)$.

\begin{lemma}\label{lem.propor}
If $F_1=\lambda F_2$ for some number $\lambda$, then ${\bf p}$
is a Dixon mechanism of the first kind.
\end{lemma}

\begin{proof}
Set $F=F_1-\lambda F_2$. This is a polynomial of the form
$\sum_{k,l=0}^4 c_{kl}\,t_1^k t_2^l$
where $c_{kl}$ are polynomials in $r_{ij}^2$ and we have
$c_{00}=c_{01}=c_{43}=c_{44}=0$.
 By hypothesis $c_{kl}$ must vanish.
 There is a symmetry $c_{4-k,4-l}=r_1^{2k-4}r_2^{2l-4}c_{k,l}$,
 hence only 11 of these 21 equations are distinct.
A computation shows that
$$
   c_{04}=r_1^4(R_1^2-r^2-\lambda R_2^2+\lambda r^2),
$$
$$
   c_{20}=r_2^4\big(r_1^2(\lambda-1)-\lambda r_{21}^2+r_{11}^2\big),\qquad
   c_{02}=r_1^4\big(r_2^2(\lambda-1)-\lambda r_{22}^2+r_{12}^2\big).
$$

{\it Case 1.} $\lambda=0$.
Then the equations $c_{04}=c_{20}=c_{02}=0$ yield $R_1=r$, $r_1=r_{11}$, and $r_2=r_{12}$,
hence the joints $p_0$ and $p_1$ are equidistant from all the $q_j$ and the
result follows from Lemma~\ref{collin}.

\smallskip
{\it Case 2.} $\lambda=1$.
Then the equations $c_{04}=c_{20}=c_{02}=0$ yield
$R_1=R_2$, $r_{11}=r_{21}$, and $r_{12}=r_{22}$,
hence the joints $p_1$ and $p_2$ are equidistant from all the $q_j$ and again the
result follows from Lemma~\ref{collin}.

\smallskip
{\it Case 3.} $R_2=r$ and $\lambda(1-\lambda)\ne 0$.
Then the equation $c_{04}=0$ implies $R_1=r$. Find $r_{11}^2$ and $r_{12}^2$
from the equations $c_{20}=c_{02}=0$ and plug the result into $c_{12}+c_{13}=0$.
We obtain the equation
$$
    \lambda(\lambda-1)r_1^2(r_{22}^2-r_2^2)^2=0,
$$
whence $r_{22}=r_2$, and the equation $c_{21}=0$ takes the form
$$
     \lambda(\lambda-1)r_2^2(r_{21}^2-r_1^2)^2=0.
$$
Thus $R_2=r$, $r_{21}=r_1$ and $r_{22}=r_2$,
hence the joints $p_0$ and $p_2$ are equidistant from all the $q_j$ and once again the
result follows from Lemma~\ref{collin}.

\smallskip
{\it Case 4.} $R_2\ne r$ and $\lambda(1-\lambda)\ne 0$.
From $c_{04}=0$ we find
$\lambda=(R_1^2-r^2)/(R_2^2-r^2)$.
Then the conditions $\lambda\ne 0$ and $\lambda\ne 1$
imply that $R_1\ne r$ and $R_1\ne R_2$.

Find $r_{11}^2$ and $r_{12}^2$ from
the equations $c_{20}=0$ and $c_{02}=0$ respectively and substitute the result
(and the found expression for $\lambda$)
in the equations $c_{12}+c_{13}=0$ and $c_{21}=0$. We obtain, respectively,
$\mu r_1^2(A+B)^2=0$ and $\mu r_2^2 AB=0$ where
$$
  \mu=r_1^2(R_1^2-R_2^2)(R_1^2-r^2)(R_2^2-r^2)^{-2},
$$
$$
     A = r^2   + r_{21}^2 - r_1^2 - R_2^2,    \quad
     B = r_1^2 + r_{22}^2 - r_2^2 - r_{21}^2, \quad
   A+B = r^2   + r_{22}^2 - r_2^2 - R_2^2.
$$
Since $\mu\ne0$, we have $AB=A+B=0$ whence $A=B=0$.
Put the expression for $a$ into
$c_{20}$ and $c_{02}$, and then replace
$R_2^2=r^2+r_{21}^2-r_1^2$ (in $c_{20}$) and $R_2^2=r^2+r_{22}^2-r_2^2$ (in $c_{02}$).
We obtain, respectively, $r^2+r_{11}^2=r_1^2+R_1^2$ and
$r^2+r_{12}^2=r_2^2+R_1^2$. These conditions together with $A=0$ and $B=0$
span all the conditions on the rod lengths in Proposition~\ref{dix.equiv}(a).
\end{proof}


\section{Complexification and compactification of the considered curves}
\label{sect.complex}

Instead of the affine coordinates $T_i$ and $t_j$
 (see \S\ref{main.idea}),
it will be more convenient for us to use the projective (homogeneous) coordinates
$(S_i:T_i)$ and $(s_j:t_j)$ running over the circles
$\{T_i{\overline T}_i=R_i^2 S_i{\overline S}_i\}$ and
$\{t_j{\overline t}_j=r_j^2 s_j{\overline s}_j\}$
in the complex projective line ${\mathbb{CP}}^1$.

In this and the next sections, $P_i$ and $Q_j$ will denote
copies of ${\mathbb{CP}}^1$ endowed with the respective coordinates.
Accordingly, $M$, $M_i$, and $C_i$ will denote the compactifications of the
complexifications of the respective algebraic sets introduced in \S\ref{main.idea}.
Namely,
$M=\{\hat f_{11}=\dots=\hat f_{22}=0\}\subset P_1\times P_2\times Q$,
$M_i=\{\hat f_{i1}=\hat f_{i2}=0\}\subset P_i\times Q$, and
$C_i=\{\hat F_i=0\}\subset Q$ where
$Q=Q_1\times Q_2={\mathbb{CP}}^1\times{\mathbb{CP}}^1$,
$$
   \hat f_{ij}(S_i,T_i;\,s_j,t_j) = S_i^2 s_j^2 f(T_i/S_i, t_j/s_j), \qquad
   \hat F_i(s_1,t_1;\,s_2,t_2) = s_1^4 s_2^4 F(t_1/s_1, t_2/s_2).
$$
We also define the curves $C_{ij} = \{\hat f_{ij} = 0\}\subset P_i\times Q_j$.

Despite the fact that we have extended $M$, we still reserve the term
$(3,3)$-framework for ``true $(3,3)$-frameworks'' only, i.e., for the elements of $M$
all whose coordinates belong to the circles
$\{T_i{\overline T}_i=R_i^2 S_i{\overline S}_i\}$ and
$\{t_j{\overline t}_j=r_j^2 s_j{\overline s}_j\}$; we denote
the set of them (i.e., ``the old $M$'') by ${\mathbb R}M$. This is the fixed point
set of the antiholomorphic involution which acts on each factor $P_i$, $Q_j$ as
\begin{equation}\label{eq.conj}
   (T_i:S_i)\mapsto (R_i^2{\overline S}_i : {\overline T}_i), \qquad
   (t_j:s_j)\mapsto (r_j^2{\overline s}_j : {\overline t}_j).
\end{equation}


\section{Consequences of the reducibility of $\hat f_{ij}$ and $\hat F_i$.}
\label{sect.reduc}

Introduce the notation as in \S\ref{sect.complex}. Assume that
$M$ contains a flexible non-overlapping $(3,3)$-framework ${\bf p}$.
In this section we find necessary conditions for the reducibility of $\hat F_1$.
Let us simplify the notation: $T=T_1$, $S=S_1$, $R=R_1$,
$$
    a_j=r_{1j}, \qquad A_0^\pm=R\pm r,\qquad A_j^\pm = a_j\pm r_j, \qquad j=1,2,
$$
(i.e., $A_j^\pm=r_{1j}\pm r_{0j}$, $j=0,1,2$).
Set $D_j=\text{Discr\,}_{t_j}(f_{1j})$. 
A computation shows that
\begin{equation}\label{dd}
   D_j = d_j^+ d_j^-, \qquad
   d_j^\pm = T^2 + \big(R^2 + r^2 - (A_j^\pm)^2\big)T + R^2,\qquad j=1,2,
\end{equation}
and for $\Delta_j^\pm = \text{Discr\,}_T(d_j^\pm)$ we have
\begin{equation}\label{Delta}
   \Delta_j^\pm = (A_j^\pm + A_0^+)(A_j^\pm - A_0^+)(A_j^\pm + A_0^-)(A_j^\pm - A_0^-).
\end{equation}
It follows from Lemma~\ref{nuiz} that
\begin{equation}\label{ineq}
       A_j^+ \pm A_k^- \ne 0, \qquad A_j^+ + A_k^+ \ne 0, \qquad j,k=0,1,2.
\end{equation}

\begin{lemma}\label{lem.bb} {\rm(Proof is obvious.)}
   If two polynomials $T^2+b_k T + R^2$, $k=1,2$, have a common root, then they coincide.
\end{lemma}

Recall that {\it deltoid} is a
4-cycle symmetric with respect to one of its diagonals, which we call in this case
the {\it axis} of the deltoid.

\begin{lemma}\label{lem.f[i,j]}
The polynomial $\hat f_{1j}$, $j=1,2$, is reducible over $\mathbb{C}$ if and only if
the 4-cycle $p_0q_0p_1q_j$ either is parallelogrammatic or it is a deltoid.%
\footnote{This statement is similar but not equivalent to \cite[Lemma 4]{\refKov}.}
\end{lemma}

\begin{proof} The reducibility in the deltoid case is evident.
For a parallelogrammatic cycle which is not a deltoid, it is also easily seen: 
the irreducible components correspond to parallelograms and antiparallelograms.
Let us prove that there are no other cases of reducibility.

Let $\hat f_{1j}$ be reducible.
Consider firstly the case when $\hat f_{1j}$ has a non-constant divisor
$\hat f_0$ of degree zero in $t_j$. Write
$\hat f_{1j}=c_2t_j^2+c_1s_jt_j+c_0s_j^2$.
Then $\hat f_0$ divides all the coefficients $c_k(S,T)$. We have
$c_0=r^2_jT(S-T)$ and $c_2=S(r^2T-R^2S)$. Hence $R=r$ and $\hat f_0=S-T$,
i.e., the polynomial $c_1$ must vanish identically in $T$ after the substitution
$R=r$, $S=T$. Performing this substitution, we obtain $c_1=(r_j^2-a_j^2)T^2$.
Hence $R=r$ and $r_j=a_j$, which corresponds to a deltoid.

Now consider the case when $\hat f_{1j}$ does not have non-constant divisors
of degree zero in $t_j$. Then $\hat f_{1j}=\hat f_1\hat f_2$,
$\deg_{t_j}\hat f_k=1$, $k=1,2$.
In this case, the discriminant $D_j$ must be a complete square.
We have $d_j^+-d_j^-=4a_jr_jT$ (see (\ref{dd})), hence
$d_j^+$ and $d_j^-$ do not coincide. This fact combined with Lemma~\ref{lem.bb} implies that $d_j^\pm$ are also complete squares, i.e., $\Delta_j^+=\Delta_j^-=0$.
Then, due to (\ref{Delta}) and (\ref{ineq}),
\begin{equation}\label{eq.f[i,j]}
   A_j^+ - A_0^+ = A_j^- - A_0^-=0
  \quad\text{ or }\quad
   A_j^+ - A_0^+ = A_j^- + A_0^-=0.
\end{equation}
Solving these systems of equations, we obtain either
$a_j=R$ and $r_j=r$ (deltoid), or
$a_j=r$ and $r_j=R$ (parallelogram).
The lemma is proven.
\end{proof}

\begin{lemma}\label{lem.transv}
  Suppose that $\hat f_{11}$ and $\hat f_{12}$ are irreducible. Then:

\begin{enumerate}
\item[(a)]
  the projection of $M_1$ to each of the factors $P_1$, $Q_1$, or $Q_2$ is finite
  (i.e., the preimage of each point is finite), and hence
  $M_1$ is an algebraic curve;
\item[(b)]
  the surfaces $\{\hat f_{1j}=0\}\subset P_1\times Q$, $j=1,2$,
  cross transversally everywhere except, maybe, a finite number of points.
\end{enumerate}
\end{lemma}

\begin{proof}

\smallskip
(a). 
Denote with $\text{pr}_j:P_1\times Q\to P_1\times Q_j$, $j=1,2$, the standard
projections. If $\text{pr}_1^{-1}(p,q) \subset M_1$, then
$\{p\}\times Q_2\subset\text{pr}_2(M_1)=C_{12}$, which contradicts the irreducibility
of $\hat f_{12}$.
Hence the projection of $M_1$ to $P_1\times Q_1$ is finite.
In the same way we prove the finiteness of the projections of
$M_1$ to $P_1\times Q_2$ and $Q$.
The finiteness of the projection of $C_{1j}$ (and hence of $M_1$) to $P_1$ and $Q_j$
is immediate from the irreducibility of $\hat f_{1j}$.

\smallskip
(b). Consider the affine chart $(T,t_1,t_2)$ on $P_1\times Q$ (the arguments
for the other affine charts are the same). In this chart, $M_1$ is defined
by the equations
$f_{11}=f_{12}=0$. The gradients have the form $\nabla f_{11}=(a,b,0)$,
$\nabla f_{12}=(c,0,d)$. If such vectors are proportional, then $b=0$ or $d=0$,
which means that one of the partial derivatives $\partial f_{1j}/\partial t_j$
is equal to zero. This may happen only on a finite number of lines of the form
$T,t_j=\text{const}$.
Due to (a), each such line crosses $M_1$ at a finite
number of points, which completes the proof.
\end{proof}

\begin{lemma}\label{lem.F[1]}
If $\hat f_{11}$ and $\hat f_{12}$ are irreducible and
$\hat F_1$ is a non-zero reducible polynomial which is not a power
of an irreducible polynomial, then the 4-cycle
$p_0q_1p_1q_2$ either is parallelogrammatic, or it is a deltoid with axis $p_0p_1$.
\end{lemma}

\begin{proof}
Recall that $C_{1j} = \{\hat f_{1j} = 0\}\subset P_1\times Q_j$.
Let $\tilde\pi_j:M_1\to C_{1j}$ and $\pi_j:C_{1j}\to P_1$ be
the standard projections $P_1\times Q\to P_1\times Q_j\to P_1$
restricted to the respective curves.
By hypothesis, the image of $M_1$ under the projection $P_1\times Q\to Q$ is
the reducible curve $C_1=\{\hat F_1=0\}$, hence the curve $M_1$ is reducible as well.
Let $M'_1$ and $M''_1$ be two distinct irreducible components of $M_1$.
By Lemma~\ref{lem.transv}, none of them can be contracted to a point by
the projections $\tilde\pi_j$.
Therefore, since these projections are two-fold
(recall that the degree of $f_{ij}$ in each variable is 2),
their restrictions to each component of $M_1$ are bijective.
Hence the composition
$$
    \eta=\pi_2\circ\tilde\pi_2\circ\big(\tilde\pi_1|_{M'_1}\big)^{-1}:
       C_{11}\to M'_1\to C_{12}\to P_1
$$
has the same branching points (the critical values) as $\pi_2$.
Since $\eta=\pi_1$, we conclude that $\pi_1$ and
$\pi_2$ have the same branching points.

The branching points of $\pi_j$ are the odd multiplicity zeros
of $D_j$ (see (\ref{dd})), hence $D_1D_2$ is a complete square.
Since $\pi_j$ is a two-fold projection of the irreducible curve $C_{1j}$,
each of $\pi_1$, $\pi_2$ has branching points. Then
$D_1$ and $D_2$ have a common root. Hence, by Lemma~\ref{lem.bb},
one of $d_1^\pm$ coincides with one of $d_2^\pm$. Note that
$d_1^+ - d_2^- = (A_2^- + A_1^+)(A_2^- - A_1^+)T$, whence $d_1^+\not\equiv d_2^-$
by (\ref{ineq}). Similarly, $d_1^-\not\equiv d_2^+$.
Hence one of the following cases takes place.

\smallskip
{\it Case 1.} $d_1^+\equiv d_2^+$ and $d_1^-\equiv d_2^-$. Since
$d_1^\pm - d_2^\pm = (A_2^\pm + A_1^\pm)(A_2^\pm - A_1^\pm)T$,
we derive from (\ref{ineq}) that
\begin{equation}\label{eqsA}
   A_2^+ - A_1^+ = A_2^- - A_1^-=0
  \quad\text{ or }\quad
   A_2^+ - A_1^+ = A_2^- + A_1^-=0.
\end{equation}
By solving these systems of equations, we obtain either
$a_1=r_2$ and $a_2=r_1$ (parallelogram),
or $a_1=a_2$ and $r_1=r_2$ (deltoid with axis $p_0p_1$).

\smallskip
{\it Case 2.} $d_1^-\equiv d_2^-$, $\Delta_1^+=\Delta_2^+=0$.
Due to (\ref{Delta}) and (\ref{ineq}),
the second condition yields $A_1^+ - A_0^+ = A_2^+ - A_0^+ = 0$.
Eliminating $A_0^+$ and factorizing (as in Case 1) $d_1^- - d_2^-$,
we again obtain (\ref{eqsA}).

\smallskip
{\it Case 3.} $d_1^+\equiv d_2^+$, $\Delta_1^-=\Delta_2^-=0$.
Due to (\ref{Delta}) and (\ref{ineq}),
the second condition yields
\begin{equation}\label{eqsA.3}
(A_1^- - A_0^-)(A_1^- + A_0^-)=(A_2^- - A_0^-)(A_2^- + A_0^-)=0,
\end{equation}
which is equivalent to four systems of linear equations. Eliminating
$A_0^-$ from each of them and combining the result with the equation
\hbox{$A_2^+ - A_1^+=0$}
(which follows from $d_1^+\equiv d_2^+$), each time we obtain one of the systems
of equations in (\ref{eqsA}). The lemma is proven.
\end{proof}

\begin{lemma}\label{lem.F[1].square}
Suppose that $\hat f_{11}$ and $\hat f_{12}$ are irreducible and
$\hat F_1=F^m$, $m\ge 1$, where $F$ is either identically zero or an irreducible 
polynomial.
Then $\bf p$ is a Dixon mechanism of the first kind.%
  \footnote{For Dixon-1 we have $F_1=(R_1^2-r^2)F^2$ and
              $F_2=(R_2^2-r^2)F^2$
            with the same $F$.}
\end{lemma}

\begin{proof} If $F=0$, this is a particular case ($\lambda=0$)
of Lemma~\ref{lem.propor}, so let $F\ne 0$.
If $m=1$ (i.e., $\hat F_1$ is irreducible), then, since $\hat F_1$ and $\hat F_2$
are bihomogeneous polynomials of the same bidegree which have a common divisor,
the result follows again from Lemma~\ref{lem.propor}.

Let $m\ge 2$. Let us prove in this case that the projection
$\pi:M_1\to C_1$ is two-fold.
Suppose that $C_1$ contains a smooth point $q$ with a single preimage.
Let $\gamma:({\mathbb C},0)\to(Q,q)$
be a holomorphic germ transverse to $C_1$.
By Lemma~\ref{lem.transv}, we may assume that the surfaces
$\hat f_{1j}=0$ are smooth and cross transversally over $q$.
Using the expression of the resultant of two polynomials via their roots
(see, e.g., \cite[Ch.~12, eq.~(1.3)]{\refGKZ}) one can easily derive that
$F_1(\gamma(t))$ has a first order zero at $t=0$. This fact contradicts
the condition $m\ge 2$, hence the projection $\pi$ cannot be one-fold.
Since $\deg_T\hat f_{ij}=2$, we conclude that it is two-fold.

Thus almost all points of $C_1$ have two preimages in $M_1$.
Since $\bf p$ is flexible, we may then assume that
$\pi^{-1}(q_1,q_2)=\{(p_1,q_1,q_2),(p'_1,q_1,q_2)\}$, $p'_1\ne p_1$.
This set itself and one of its elements are invariant under the antiholomorphic
involution (\ref{eq.conj}), hence the other element is invariant as well.
Therefore $(p'_1,p_2;\,q_1,q_2)\in{\mathbb R}M$. Moreover, all this remains true
during a deformation of $\bf p$. Hence the $(4,3)$-framework
$(p_0,p_1,p_2,p'_1;\,q_0,q_1,q_2)$ is flexible. Its joints
$p_1$ and $p'_1$ are equidistant from all the $q_j$'s. With help of
Lemma~\ref{collin}, it is easy to derive from this fact that
$\bf p$ is a Dixon mechanism of the first kind.
The lemma is proven.
\end{proof}

Recall our assumption that $M$ contains a flexible non-overlapping
$(3,3)$-framework
${\bf p}$. 
Say that a cycle in $\bf p$ is {\it fastened}, if it contains the edge $p_0q_0$.
One can summarize Lemmas \ref{lem.f[i,j]}, \ref{lem.F[1]}, and \ref{lem.F[1].square}
as follows.

\begin{lemma}\label{main.lemma}
{\rm(Main Lemma.)}
If $\bf p$ is not a Dixon mechanism of the first kind,
then the $(2,3)$-framework $(p_0,p_1;\,q_0,q_1,q_2)$ contains either
a parallelogrammatic cycle, or a fastened deltoid, or a not fastened deltoid with
axis $p_0p_1$.
\end{lemma}


\section{Completing the proof of Theorem~\ref{t1}}\label{sect.EOP}

Let $\bf p$ be a flexible non-overlapping $(3,3)$-framework which is not
a Dixon mechanism of the first kind. Let us show that $\bf p$ is
a Dixon mechanism of the second kind.

\begin{lemma}\label{lem.deltoid}
Any deltoid in $\bf p$ is a rhombus.
\end{lemma}

\begin{proof}
Suppose that $\bf p$ contains a deltoid $\Delta$ which is not a rhombus.
Renumber the joints so that
$\Delta=p_0q_1p_1q_2$ and $|p_0q_1|=|p_1q_1|\ne|p_0q_2|=|p_1q_2|$
(see Fig.~\ref{fig.lem}, on the left).
By Lemma~\ref{main.lemma}, the $(2,3)$-framework 
$(p_0,p_1;\,q_0,q_1,q_2)$ must contain a 4-cycle
$\Delta'$ realizing one of the following cases.
In each of them (except the last one) we show that
$p_0$ and $p_1$ are equidistant from $q_0$, which contradicts Lemma~\ref{collin}.

\smallskip
{\it Case 1. Parallelogrammatic cycle.} Then
$\Delta'$ contains both $p_0$, $p_1$, and also at least one of
$q_i$, $i=1$ or $2$. Since $|p_0q_i|=|p_1q_i|$, we conclude that
$\Delta'$ is a rhombus. But $\Delta'\ne\Delta$ (since $\Delta$ is not a rhombus),
hence $q_0\in\Delta'$. Therefore $p_0$ and $p_1$ are equidistant from $q_0$.  

\smallskip
{\it Case 2. Fastened deltoid with axis $p_0p_1$.}
We may assume that
$\Delta'=p_0q_0p_1q_1$, $|p_0q_0|=|p_0q_1|$ and $|p_1q_1|=|p_1q_0|$.
Then $|p_0q_0|=|p_0q_1|=|p_1q_1|=|p_1q_0|$.

\smallskip
{\it Case 3. Fastened deltoid with axis $q_0q_j$.}
By definition $|p_0q_0|=|p_1q_0|$.

\smallskip
{\it Case 4. Not fastened deltoid with axis $p_0p_1$.}
Then $\Delta'=\Delta$ and this is a deltoid with two axes, that is a rhombus.
A contradiction. The lemma is proven.
\end{proof}

\begin{figure}[ht]
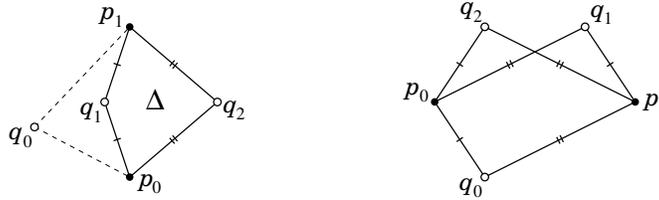

\centering
\includegraphics[width=32.17 mm]{lem-deltoid.eps}\hskip 20mm
\includegraphics[width=36.67 mm]{lem-parall.eps}
\caption{Illustration to the proofs of Lemmas
          \ref{lem.deltoid} (on the left) and \ref{lem.paral} (on the right).}
\label{fig.lem}
\end{figure}

\begin{lemma}\label{lem.paral}
$\bf p$ cannot contain two distinct parallelogrammatic cycles with three common vertices.
\end{lemma}

\begin{proof} Suppose that $\bf p$ contains
two distinct parallelogrammatic cycles $\Pi_1$ and $\Pi_2$ with three common vertices.
Up to renumbering, we may assume that these are $q_0p_0q_1p_1$ and $p_0q_1p_1q_2$
(see Fig.~\ref{fig.lem}, on the right). Then $q_0p_0q_2p_1$ is a deltoid.
By Lemma~\ref{lem.deltoid}, it must be a rhombus. Hence
$\Pi_1$ and $\Pi_2$ are rhombi as well. It is easy to check that this is impossible.
The lemma is proven.
\end{proof}

Lemmas~\ref{main.lemma} and \ref{lem.deltoid} imply that each
$(2,3)$-framework obtained from $\bf p$ by removal of one joint
contains a parallelogrammatic 4-cycle. 
Using Lemma~\ref{lem.paral}, it is easy to derive from this fact that
the joints of $\bf p$ can be numbered so that the three 4-cycles
$\Pi_{ij}=p_iq_ip_jq_j$, $i<j$, become parallelogrammatic. This means that
one can denote the lengths of the rods by $a,b,c,d$ as in
Proposition~\ref{dix.equiv}(b). It remains to prove that the relation
$a^2+c^2=b^2+d^2$ holds up to renumbering of the joints.
In the notation of~\S\ref{main.idea} we have
$$
   r=r_{11}=r_{22}=a, \qquad R_1 = r_1 = b, \qquad R_2=r_2=c, \qquad r_{12}=r_{21}=d.
$$
Doing these substitutions, we express the coefficients of $F_1$ and $F_2$
as polynomials in $a,b,c,d$.
By Lemma~\ref{lem.res}, the resultant of $F_1$ and $F_2$ with respect to $t_1$
identically vanishes. Hence the resultant of
$F_1(t_1,-c/a)$ and $F_2(t_1,-c/a)$ is zero. A computation shows that it is equal to
\begin{equation}\label{res.bcd}
 16\, b^8 c^{16} (a + c)^8 (a^2 + b^2 - c^2 - d^2)^4 (a^2 + d^2 - b^2 - c^2)^4
 (a^2 + c^2 - b^2 - d^2)^4,
\end{equation}
which completes the proof of Theorem~\ref{t1}.


\section{Flexibility of hyperbolic bipartite frameworks}\label{sect.hyp}

\def\distH{d_H}

Let $H^2$ be the standard hyperbolic plane, i.e.,
a complete simply connected riemannian 2-manifold of constant curvature
equal to $-1$.
We denote the distance in $H^2$ by $\distH(\,,\,)$.
The flexibility condition (\DixI) extends without changes to
the hyperbolic case. Condition (\DixII) admits the following equivalent
reformulation, which also extents to the hyperbolic case:

\smallskip
(\DixII) There are two orthogonal lines
and two quadrilaterals symmetric with respect to each of them
and with vertices not belonging to them such that
$p_1,\dots,p_m$ are at the vertices of one quadrilateral and
$q_1,\dots,q_n$ are at the vertices of the other one.

\begin{theorem}\label{tH} Theorem~\ref{t1} holds for $H^2$.
\end{theorem}

The proof of Theorem~\ref{tH} is almost the same as for Theorem~\ref{t1}.
In this section we just explain which elements of the proof (mostly, the formulas)
should be modified.


\subsection{Lobachevsky coordinates in $H^2$.
Hyperbolic version of \S1 and \S\ref{sect.prelim}}
\label{sect.Loba}

It is evident that Lemmas~\ref{fix} and \ref{nuiz} are valid for $H^2$.
For other facts from \S1 and \S\ref{sect.prelim},
it is convenient to use the following
hyperbolic analog of the Cartesian coordinate system called {\it Lobachevsky
coordinate system}. Fix an oriented line $\ell$ and a point $O\in\ell$.
Then the coordinates $(x,y)$ of a point $p$ are
$x=\pm\distH(O,q)$ and $y=\pm\distH(p,q)$ where $q\in\ell$ is such that
$pq\perp\ell$ and the signs are chosen according to the quadrant containing $p$.
In these coordinates we have
\begin{equation}\label{L.coord}
      \cosh\distH\big((x_1,y_1),(x_2,y_2)\big)
      = \cosh y_1\cosh y_2\cosh(x_2-x_1) - \sinh y_1\sinh y_2.
\end{equation}
The following is a hyperbolic analog of Lemma~\ref{4kd}.

\begin{lemma}\label{4kd.hyp}
The diagonals of a quadrilateral (maybe, self-crossing)
are orthogonal if and only if $\cosh a \cosh c=\cosh b \cosh d$ where $a,b,c,d$
are the lengths of its consecutive sides.
\end{lemma}

\begin{proof}
Consider a quadrilateral $p_1p_2p_3p_4$ with $\distH(p_1,p_2)=a$,
$\distH(p_2,p_3)=b$, $\distH(p_3,p_4)=c$, $\distH(p_4,p_1)=d$. 
Introduce a Lobachevsky coordinate
system with $x$-axis $p_1p_3$. Let $(x_k,y_k)$ be the coordinates of $p_k$
(then $y_1=y_3=0$). By (\ref{L.coord}) we have 
$$
   \cosh\distH(p_i, p_j)=\cosh(x_i-x_j)\cosh y_j,\qquad i=1,3,\;\;j=2,4,
$$
whence the following identity, which implies the result:
$$
  \cosh a\cosh c - \cosh b\cosh d=\sinh(x_1-x_3)\sinh(x_4-x_2)\cosh y_2\cosh y_4.
$$
\end{proof}

For the sake of coherence with the Euclidean case,
we still say that a $4$-cycle is {\it parallelogrammatic}
if the opposite sides have equal lengths (though parallelism no longer plays any role).
We call it {\it anti-parallelogram} (resp. {\it parallelogram}) either if it is
{\it degenerate}, i.e., all its vertices are collinear, or if it is (resp. is not)
self-crossing.

The following is a hyperbolic analog of Proposition~\ref{dix.equiv} and the proof
is also similar.

\begin{proposition}\label{propH}
Let ${\bf p}=(p_0,\dots,p_{m-1};\,q_0,\dots,q_{n-1})$ be a non-overlapping
$(m,n)$-framework in $H^2$. Denote $u_{ij}=\cosh\distH(p_i,q_j)$.

\smallskip
(a). $\bf p$ satisfies (\DixI) if and only if,
for each cycle $p_iq_jp_kq_l$, one has $u_{ij}u_{kl}=u_{il}u_{jk}$.
As in Proposition~\ref{dix.equiv},
these conditions for cycles with fixed $i$ and $j$ generate all the others.

\smallskip
(b). If $\bf p$ is flexible and $m=n=3$, then $\bf p$ satisfies (\DixII) if and only if,
up to renumbering, one has
$u_{00}=u_{11}=u_{22}$, $u_{01}=u_{10}$, $u_{12}=u_{21}$, $u_{20}=u_{02}$, and
$u_{00}+u_{02} = u_{01} + u_{12}$.
\end{proposition}

\begin{proof}
(a). Immediately follows from Lemma~\ref{4kd.hyp}.

(b). The condition on the lengths is derived from (\DixII) by a direct computation
in Lobachevsky's coordinates.
Let us prove the inverse implication. Let $\bf p$ be flexible and
satisfy the condition on the lengths.
Consider a smooth deformation ${\bf p}={\bf p}(t)$ with constant $u_{ij}$'s.
The cycles $\Pi_{ij}=p_iq_ip_jq_j$ are
parallelogrammatic. Suppose that $\Pi_{01}$ and $\Pi_{12}$ are parallelograms.
We choose Lobachevsky's coordinates so that the $x$-axis is the line passing through
the centers of symmetry of $\Pi_{01}$ and $\Pi_{12}$ (we may assume that
this condition is fulfilled for each $t$). Then the composition
of these central symmetries is a shift $(x,y)\mapsto(x+a,y)$ such that
$p_0\mapsto p_2$, $q_0\mapsto q_2$. Since $u_{02}=u_{20}$, this fact combined
with (\ref{L.coord}) implies that
$p_0$ and $q_0$ (as well as $p_2$ and $q_2$) have equal $x$-coordinates.
Hence, we have $p_k=(x_k,(-1)^k y_p)$, $q_k=(x_k,(-1)^k y_q)$, $k=0,1,2$. 
By a shift of the $x$-coordinate
we can achieve $x_1(t)=0$ for each $t$. Then, by (\ref{L.coord}), we have
$$
\begin{matrix}
   u_{00}=\cosh(y_p-y_q), &
   u_{01}+u_{00}=\cosh y_p\cosh y_q(\cosh x_0+1), \\
   u_{02}-u_{00}=\cosh y_p\cosh y_q(\cosh(x_2-x_0)-1), &
   u_{12}+u_{00}=\cosh y_p\cosh y_q(\cosh x_2+1).
\end{matrix}
$$
Differentiating these identities with respect to $t$, we obtain four linear
homogeneous equations for $x'_0,x'_2,y'_p,y'_q$ (cf.~the proof of Lemma~\ref{collin}).
The determinant is equal to
$$
  (1+\cosh x_0)(1+\cosh x_2)\big(\cosh y_p\cosh y_q\big)^2
  \big(1-\cosh(x_2-x_0)\big)
  \sinh(y_p+y_q)\sinh(y_p-y_q).
$$
It vanishes only when $y_p+y_q=0$ (since $\bf p$ is non-overlapping), which means
that $\bf p$ is symmetric with respect to the $x$-axis. However,
this condition cannot be kept during a non-constant deformation.

The obtained contradiction shows that $\Pi_{01}$ or $\Pi_{12}$ is an
anti-parallelogram. Let it be $\Pi_{01}$ (the case of $\Pi_{12}$ is analogous).
Then we can choose Lobachevsky's coordinates so that
$p_0=(x_p,y_p)$, $q_0=(x_q,y_q)$, $p_1=(x_p,-y_p)$, $q_1=(x_q,-y_q)$.
Shifting the $x$-coordinates,
we can achieve that $p_2=(-x_p,-y)$ for some $y\in\mathbb R$.
Then by (\ref{L.coord}) we have
$$
    u_{00} + u_{02} - u_{01} - u_{12} = 2(\sinh y - \sinh y_p)\sinh y_q
$$
whence $p_2=(-x_p,-y_p)$. Then $q_2=(-x_q,-y_q)$ because $q_2$ is uniquely
determined by the distances to the three non-collinear points $p_0,p_1,p_2$.
\end{proof}

Using Lobachevsky's coordinates, the proof
of Lemma~\ref{collin} repeats word-by-word in the hyperbolic setting but
the identities (\ref{eq.collin}) should be replaced with
$$
   \frac{\cosh x_j}{\cosh x_0} = \frac{\cosh r_j}{\cosh r_0},\qquad
   \frac{\cosh(x_j-a)}{\cosh(x_0-a)} = \frac{\cosh R_j}{\cosh R_0},\qquad j=1,2,
$$
and then the determinant of the linear system for
$x'_0,x'_1,x'_2,a'$ becomes
$$
   \frac{\sinh a\, \sinh(x_0-x_1)\,\sinh(x_0-x_2)\,\sinh(x_1-x_2)}
    {\cosh x_0\cosh x_1\cosh x_2\cosh(x_0-a)\cosh(x_1-a)\cosh(x_2-a)}.
$$


\subsection{Poincar\'e model of $H^2$.
Hyperbolic version of \S\ref{main.idea} and \S\S\ref{sect.gen}--\ref{sect.EOP}}
\label{sect.Poin}

In this subsection
we use the Poincar\'e model of $H^2$ in the unit disk
$\mathbb D\subset\mathbb C$, where the geodesics are circles orthogonal to
$\partial\mathbb D$, and the distance is
\begin{equation}\label{distH}
   \cosh\distH(p,q) = 1+\frac{2(p-q)(\bar p-\bar q)}{(1-p\bar p)(1-q\bar q)},
\end{equation}
in particular the $\distH$-circle of radius $r$ centered at $0$ is the
$\mathbb C$-circle
$\{|z|=l\}$ where $u=\cosh r$, $l=\rho_H(u)$, and the function
$\rho_H:[1,+\infty)\mapsto[0,1)$ is defined by $\rho_H(u)=\sqrt{(u-1)/(u+1)}$.

We still denote the rod lengths $r_{ij}=\distH(p_i,q_j)$, $R_i=r_{i0}$, $r_j=r_{0j}$,
$r=r_{00}$.
We also set
\begin{equation}\label{u[i,j]}
\begin{matrix}
 u_{ij}=\cosh r_{ij}, & U_i=\cosh R_i, & u_j=\cosh r_j, & u=\cosh r,\\
 l_{ij}=\rho_H(u_{ij}), & L_i=\rho_H(U_i), & l_j=\rho_H(u_j), & l=\rho_H(u).
\end{matrix}
\end{equation}
Let $q_0=0$ and $p_0=l$ (then $\distH(p_0,q_0)=r$). For the circles
$$
   P_k=\{p_k\in\mathbb D\mid\distH(q_0,p_k)=R_k\},\qquad
   Q_k=\{q_k\in\mathbb D\mid\distH(p_0,q_k)=r_k\},\qquad k=1,2,
$$
we choose the parametrizations
$p_i(T_i) = l T_i$ and $q_j(t_j) = (l+lt_j)/(l^2t_j+1)$ where the parameters
$T_i$ and $t_j$ run over the circles $|T_i|=L_i/l$ and $|t_j|=l_j/l$
respectively. In order to check that $t_j\mapsto q_j(t_j)$ parametrizes $Q_j$,
remark that $Q_j$ is the image of the circle
$$
   Q_j^*=\{q_j^*\in\mathbb D\mid\distH(q_0,q_j^*)=r_j\}
        =\{|z|=l_j\}
$$
under the mapping
$z\mapsto(l+z)/(lz+1)$, which is a conformal isomorphism of $\mathbb D$
taking $0$ to $l$ (i.e, taking $q_0$ to $p_0$).
We define the algebraic sets $M$, $M_i$, $C_i$, and $C_{ij}$
as in the Euclidean case.
Then the curve $C_{ij}$ has defining equation $f_{ij}(T_i,t_j)=0$ where
$f_{ij}$ is the numerator of the rational function in $T_i,t_j$ obtained from
$\distH\big(p_i(T_i),q_j(t_j)\big)-u_{ij}$ by applying
(\ref{distH}) with the substitutions
$\overline T_i=L_i^2/(l^2T_i)$ and $\overline t_j=l_j^2/(l^2t_j)$.
So, we may define $f_{ij}$ by setting
$$
  \frac{f_{ij}(T_i,t_j)}{4(u+1)T_it_j} =
  1 + \frac{ 2\Big(lT_i-q_j(t_j)\Big)
           \left( \frac{L_i^2}{lT_i} - q_j\big(\frac{l_j^2}{l^2t_j}\big) \right)}
      {(1-L_i^2)\left(1- q_j(t_j)q_j\big(\frac{l_j^2}{l^2t_j}\big) \right)}
      - u_{ij}
$$
($f_{ij}$ is invariant under $T_i\leftrightarrow-t_j$,
$L_i\leftrightarrow l_j$ though it is not immediately seen in this formula).
We see that $f_{ij}$ is a polynomial in $T_i,t_j$ of degree 2 in each variable;
its coefficients are rational functions of $l^2,L_i^2,l_j^2,u,u_{ij}$.
By the substitution $l=\rho_H(u)$, $L_i=\rho_H(U_i)$, $l_j=\rho_H(u_j)$, we
express $f_{ij}$ as a sum of 72 monomials in $T_i,t_j,u,U_i,u_j,u_{ij}$, $i,j=1,2$.
We set
$$
      F_i(t_1,t_2) = \frac{\text{Res}_{\,T_i}(f_{i1},f_{i2})}{16(1+u)^4(1-U_i^2)}.
$$
It is a sum of $445$ monomials in $t_1$, $t_2$, and all the $u_{ij}$.
As in \S\ref{main.idea}, $\deg_{t_j}F_i=4$ for each $i,j$.

Below we use the notation $A\doteq B$ to say that $A=n\mu B$ where $n\in\mathbb Q$
and $\mu$ is a product of some factors of the form $(u_{ij}\pm 1)^{\pm1}$, $i,j=0,1,2$.

\begin{proof}[Proof of Lemma~\ref{lem.propor} in the hyperbolic setting]
Let $c_{kl}$ be the coefficient of $t_1^k t_2^l$ in $F_1-\lambda F_2$.
We have
$$
   c_{04} 
    \doteq U_1^2-u^2-\lambda U_2^2+\lambda u^2,
$$
$$
   c_{20} 
    \doteq u_1^2(\lambda-1)-\lambda u_{21}^2+u_{11}^2,\qquad
   c_{02} 
    \doteq u_2^2(\lambda-1)-\lambda u_{22}^2+u_{12}^2,
$$
thus, if $\lambda(1-\lambda)=0$, the arguments as in
the proof of Lemma~\ref{lem.propor} yield the result. 
Assume that $\lambda(1-\lambda)\ne 0$.
Let $\Lambda=\mathbb Q[\lambda,u_{ij}]_{i,j=0,1,2}$
be the ring of polynomials in $\lambda$
and all the $u_{ij}$'s.
The coefficients $c_{kl}$ are represented by elements of $\Lambda$.
Let $e_{kl}$ be obtained from $c_{kl}$ by factorizing it in $\Lambda$ and getting rid
of all factors of the form $u_{ij}\pm 1$.
We are going to show that any real solution of the system of equations
$c_{kl}=0$, $k,l=0,\dots,4$,
such that $u_{ij}>1$ and $\lambda(1-\lambda)\ne0$ is a solution of 
the system of equations $b_{ij}=0$, $i,j=1,2$ where
$b_{ij} = u_{00}u_{ij}-u_{i0}u_{0j}$.
To this end it is enough to show that the ideal
$$
  \big(e_{02},\,e_{20},\,e_{03},\,e_{30},\,e_{04},\,
  1+w_0\lambda(1-\lambda),\,
   1+w_1 b_{11}+w_2 b_{12}+w_3 b_{21}+w_4 b_{22}\big)
$$
in $\Lambda[w_0,\dots,w_4]$ contains $1$. This fact can be checked
by computing the Gr\"obner basis (which is very fast in this case with a computer).
\end{proof}

In the proof of Lemma~\ref{lem.f[i,j]}, in the case when $\hat f_{1j}$
has a non-constant divisor $\hat f_0$ of degree zero in $t_j$, we have
$\text{Res}_{\,T}(c_0,c_2)\doteq(U_1^2-u^2)S^4$.
Hence $U_1=u$, which gives
\begin{equation}\label{c0c2}
  c_0\doteq(S-T)\big((1-u)S+(1+u)T\big),\qquad
  c_2\doteq(S-T)\big((1+u)S+(1-u)T\big).
\end{equation}
Thus $\hat f_0=S-T$, and after the substitution
$S=T$, $U_1=u$ we obtain $c_1\doteq(u_1-u_{1j})T^2$.

The rest of \S\ref{sect.reduc} repeats word-by-word using the following equalities,
where we 
set  (as in \S\ref{sect.reduc}) $A_j^\pm=r_{1j}\pm r_{0j}$, $j=0,1,2$,
and
use the notation ${\mathfrak s}(x)$ for $\sinh(x/2)$:
\begin{align}
   & d_j^\pm 
       \doteq
         (u-1)(U_1+1)T^2 + 2\big(u_j u_{1j}-u U_1
     \pm{\mathfrak s}(2r_j){\mathfrak s}(2r_{1j})\big)T + (u+1)(U_1-1),    \label{d[j]}
\\
   &{\Delta_j^\pm} 
       \doteq 
           {\mathfrak s}(A_j^\pm + A_0^+)\,{\mathfrak s}(A_j^\pm - A_0^+)\,
           {\mathfrak s}(A_j^\pm + A_0^-)\,{\mathfrak s}(A_j^\pm - A_0^-), \label{De[j]}
\\
   & d_2^{\pm_2} - d_1^{\pm_1} 
          \doteq
                {\mathfrak s}(A_2^{\pm_2} + A_1^{\pm_1})\,
                {\mathfrak s}(A_2^{\pm_2} - A_1^{\pm_1})\,T
       \qquad\text{($\pm_1$ and $\pm_2$ are independent signs)}.
                                                                        \label{d[1]-d[2]}
\end{align}

The hyperbolic version of the computation at the end of \S\ref{sect.EOP} is as
follows: if
$$
   u_{11}=u_{22}=u,\qquad U_1=u_1,\qquad U_2=u_2,\qquad u_{12}=u_{21},
$$
then the resultant of $lF_1(t_1,\pm l_2/l)$ and $lF_2(t_1,\pm l_2/l)$
with respect to $t_1$ is equal to
%
\begin{equation}\label{last.res}
\begin{matrix}
 (l\pm l_2)^8(l \hskip 1pt 
      l_2\pm\varepsilon)^8
 (u+1)^{16}(u_1^2-1)^4(u_2-1)^8(u+u_1+u_2+u_{12})^4 \hbox to 10mm{}\\
 \hbox to 14mm{}\times(u+u_1-u_2-u_{12})^4(u+u_2-u_1-u_{12})^4(u+u_{12}-u_1-u_2)^4,
  \quad\varepsilon=1
\end{matrix}
\end{equation}
(here ``$+$'' for ``$\pm$'' is enough, but both signs
will be needed at the end of \S\ref{sect.EOPS}).

\section{Flexibility of spherical bipartite frameworks}\label{sect.sph}

{\it Spherical $(m,n)$-frameworks} and their flexibility are defined as
in the planar case but the points $p_i$ and $q_j$ are chosen on the
unit sphere $\{x^2+y^2+z^2=1\}\subset{\mathbb R}^3$
so that $p_i\ne\pm q_j$ for any $i,j$.
Say that a spherical $(m,n)$-framework is {\it $\mathbb{P}^2$-non-overlapping}
if $p_i\ne\pm p_j$ and $q_i\ne\pm q_j$ when $i\ne j$.
The flexibility conditions (D1) and (D2) repeat almost word-by-word
in the spherical case (cf.~\cite[\S6]{\refWu}). They can be formulated as follows.
\begin{enumerate}
\item[(\SDixI)] $p_1,\dots,p_n$ lie on one plane, $q_1,\dots,q_m$ lie
          on another plane, and these planes are orthogonal to each other and pass
          through the origin.

\smallskip
\item[(\SDixII)] There are two orthogonal planes passing through the origin
and two rectangles symmetric
with respect to each of them and with vertices not belonging to them
such that $p_1,\dots,p_m$ are at the vertices of one rectangle
 and  $q_1,\dots,q_n$ are at the vertices of the other one.
\end{enumerate}

For $m=n=3$,
the following spherical analog of Theorem~\ref{t1}
is proven in 
\cite{\refGGLS}.

\begin{theorem}\label{tS}
Let $\min(m,n)\ge 3$.
Let $\bf p$ be a $\mathbb{P}^2$-non-overlapping spherical $(m,n)$-framework.
Then $\bf p$ is flexible if and only if
it satisfies either (\SDixI) or one of the following conditions:

\begin{enumerate}
\item[(\PDixII)] $\bf p$ satisfies (\SDixII) after
          applying the antipodal involution to some joints;

\smallskip
\item[(\DA)] $\langle p_1,q_1\rangle=\langle q_1,p_2\rangle
     =\langle p_2,q_2\rangle=-\langle q_2,p_1\rangle$ and
     $\langle p_0,q_k\rangle=\langle q_0,p_k\rangle=0$, $k=1,2$,
where $m=n=3$, ${\bf p}=(p_0,p_1,p_2;\,q_0,q_1,q_2)$, and
$\langle\,,\,\rangle$ is the scalar product in ${\mathbb R}^3$.
\end{enumerate}
\end{theorem}

The paradoxical motion of $\bf p$ in case (\DA) is called in 
\cite{\refGGLS} {\it Constant Diagonal Angle motion}.

\begin{proof}[Proof of Theorem \ref{tS}
under the assumption that it holds for $m=n=3$.]
The proof is almost the same as that of the reduction of Theorem~\ref{t1} to
its case $m=n=3$ (see \S1), except that Lemmas~\ref{sph1} and \ref{sph2}
should be taken into account.
\end{proof}

Following \cite{\refGGLS}, say that points
on a sphere are {\it cocircular} if they lie on a geodesic circle.

\begin{lemma}\label{sph1}
If a $\mathbb{P}^2$-non-overlapping spherical $(3,3)$-framework
satisfies (\PDixII) or (\DA), then the points of any part
 are not cocircular.
\end{lemma}

\begin{proof} The case of (\PDixII) is evident.
Suppose that $\bf p$ satisfies (\DA) but $p_0,p_1,p_2$ are cocircular.
Then we can choose the coordinates so that
\hbox{$p_0=(1,0,0)$}, $p_1=(x_1,y_1,0)$, $p_2=(x_2,y_2,0)$.
Then $\langle p_0,q_1\rangle=\langle p_0,q_2\rangle=0$ implies
$q_1=(0,y_3,z_3)$, $q_2=(0,y_4,z_4)$ whereas
$\langle q_0,p_1\rangle=\langle q_0,p_2\rangle=0$ implies $q_0=(0,0,\pm1)$
and we may assume $q_0=(0,0,1)$.
The conditions on $\langle p_i,q_j\rangle$, $j=1,2$,
read $y_1y_3=y_2y_3=y_2y_4=-y_1y_4$, hence
$(y_1-y_2)y_3=(y_1+y_2)y_4=0$, i.e.,
either $y_3y_4=0$ or $y_1=y_2=0$.
If $y_3y_4=0$, then $q_0=\pm q_1$ or $q_0=\pm q_2$.
If $y_1=y_2=0$, then $p_0=\pm p_1$. Both cases
are impossible for $\mathbb{P}^2$-non-overlapping frameworks.
\end{proof}

\begin{lemma}\label{sph2}
If a $\mathbb{P}^2$-non-overlapping spherical $(3,3)$-framework 
$\bf p$ satisfies (\DA) and a $\mathbb{P}^2$-non-overlapping framework
${\bf p}'$ is obtained from $\bf p$ by replacing some $q_i$ with
$q'_i\ne\pm q_i$, then ${\bf p}'$ does not satisfy
any of the conditions (\SDixI), (\PDixII), (\DA).
\end{lemma}

\begin{proof} $\bf p'$ does not satisfy (\SDixI) by Lemma~\ref{sph1}.
It does not satisfy (\DA) because any five joints of a framework satisfying (\DA)
uniquely determine the sixth one up to antipodal involution.
Let us show that ${\bf p}'$ does not satisfy (\PDixII). Suppose it does.
Condition (\DA) is invariant under applying the antipodal involution
to any joint. Hence we may assume that ${\bf p}'$ satisfies (\SDixII) while
$\bf p$ still satisfies (\DA).
Then one can choose coordinates so that each part of ${\bf p}'$ sits at the
vertices of a rectangle invariant under the reflections
$\xi:(x,y,z)\mapsto(-x,y,z)$ and $\eta:(x,y,z)\mapsto(x,-y,z)$.

If $q_i=q_0$, then $q_2$ is the image of $q_1$ under
$\xi$, $\eta$, or $\xi\eta$.
The condition $\langle p_0q_1\rangle=\langle p_0q_2\rangle=0$ then implies
that $p_0$ belongs to $\{y=0\}$, $\{x=0\}$, or $\{z=0\}$ respectively
(see Figure~\ref{fig.lem.sph2}).
This contradicts the conditions that $\bf p$ is $\mathbb{P}^2$-non-overlapping.

If $q_i\ne q_0$, the arguments are the same but
with $q_0,p_1,p_2$ instead of $p_0,q_1,q_2$.
\end{proof}

\begin{figure}[ht]
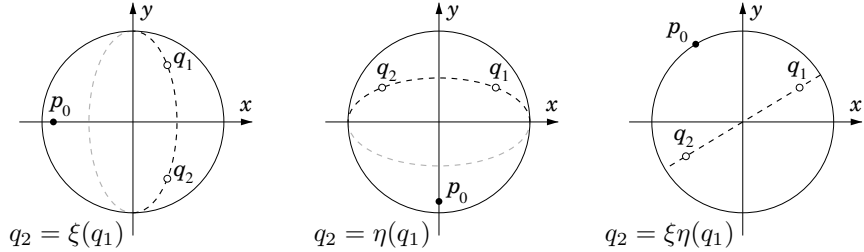

\centerline{
$q_2=\xi(q_1)$\hskip-40pt\includegraphics[width=31 mm]{lem-sph2x.eps}\hbox to 7mm{}
$q_2=\eta(q_1)$\hskip-40pt\includegraphics[width=31 mm]{lem-sph2y.eps}\hbox to 5mm{}
$q_2=\xi\eta(q_1)$\hskip-40pt\includegraphics[width=31 mm]{lem-sph2xy.eps}
}
\caption{The projection onto the $xy$-plane in the proof of Lemma~\ref{sph2}.}
\label{fig.lem.sph2}
\end{figure}

\begin{remark}
Examples as in Remark~\ref{z3} can also be constructed in the
spherical case.
\end{remark}

Below we give a proof of Theorem~\ref{tS} which is an adaptation of the proof
of Theorem~\ref{tH}. All the computations are exactly the same (just with
$\cosh x$ and $\sinh x$ replaced by $\cos x$ and $\sqrt{-1}\sin x$).
However there are more cases to consider because of the antipodal involution,
which can be applied to any joint.


\subsection{Geographic coordinates.
            Spherical version of \S\ref{sect.Loba}} \label{sect.geo}

We define the distance $d_S$ on $S^2\in\mathbb R^3$ as the length of
the shortest geodesic: $d_S(p,q)=\arccos\langle p,q\rangle$.
Lobachevsky's coordinate system in $H^2$ is a hyperbolic analog of the
usual geographic coordinates on the unit sphere:
$x$ (longitude) and $y\in[-\pi/2,\pi/2]$ (latitude). So,
$(x,y)$ are the geographic coordinates of the point
$(\cos x\cos y,\,\sin x\cos y,\,\sin y)$. In these coordinates,
\begin{equation}\label{G.coord}
  \cos d_S\big((x_1,y_1),(x_2,y_2)\big)
   =\cos y_1\cos y_2 \cos(x_2-x_1)+\sin y_1\sin y_2.
\end{equation}
Almost everything in \S\ref{sect.Loba} becomes true if one replaces
$\distH$, $\cosh$, $\sinh$, ``Lobachevsky's coordinates'',
``collinear'', ``non-overlapping'' respectively by
$d_S$, $\cos$, $\sqrt{-1}\sin$, ``geographic coordinates'',
``cocircular'', ``$\mathbb P^2$-non-overlapping''
(compare, for example, (\ref{L.coord}) with (\ref{G.coord})).
The only difference is the following.

\begin{remark}\label{diff.H.S}
In Proposition~\ref{propH}(a) for $S^2$, it is wrong in general that the conditions
$u_{ik}u_{jl}=u_{il}u_{jk}$ with fixed $i$ and $j$ generate all the other conditions
(for example, when $p_i$ is at the North Pole and all the other joints are on
the equator). However, this is true when $u_{ij}\ne 0$.
\end{remark}


\subsection{Stereographic projection onto $\mathbb C$.
  Spherical version of \S\ref{sect.Poin}
} \label{sect.stereo}

While Loba\-chevsky coordinate system is an analog of the geographic coordinates,
the Poincar\'e model is an analog of the stereographic projection
$S^2\to\mathbb C\cup\{\infty\}$ (in fact, the Poincar\'e model is 
the stereographic projection of a hyperboloid in $\mathbb R^3$ endowed
with the Minkowski $(2,1)$-distance).

The stereographic projection identifies $(x,y,z)\in S^2$ with
$(x+y\sqrt{-1}\,)/(1-z)\in\mathbb C\cup\{\infty\}$. Under this identification
we have (cf.~(\ref{distH}))
%
\begin{equation}\label{distS}
      \cos d_S(p,q) = 1 - \frac{2(p-q)(\bar{p}-\bar{q})}
                {(1+p\bar{p})(1+q\bar{q})},
\end{equation}
in particular the $d_S$-circle of radius $r$ centered at $0$ is the
$\mathbb C$-circle $\{|z| = l\}$ where $u = \cos r$,
$l = \rho_S(u)$, and the function
$\rho_S : [-1,1]\to[0,+\infty]$ is defined by $\rho_S(u) = \sqrt{(1-u)/(1+u)}$.

As in \S\ref{sect.Poin}, we set $r_{ij}=d_S(p_i,q_j)$, $R_i=r_{i0}$, $r_j=r_{0j}$,
$r=r_{00}$, and we set (cf.~(\ref{u[i,j]}))
\begin{equation}\label{u[i,j]S}
\begin{matrix}
 u_{ij}=\cos r_{ij}, & U_i=\cos R_i, & u_j=\cos r_j, & u=\cos r,\\
 l_{ij}=\rho_S(u_{ij}), & L_i=\rho_S(U_i), & l_j=\rho_S(u_j), & l=\rho_S(u).
\end{matrix}
\end{equation}
We define the circles $P_i$, $Q_j$, their parametrizations, and the polynomials
$f_{ij}$ and $F_i$ by the same formulas
as in \S\ref{sect.Poin} but with $d_S$ instead of $d_H$ and with
$q_j(t_j) = (l+lt_j)/(1-l^2_j)$.
It turns out that the expressions of $f_{ij}$ and $F_i$ in terms of $u_{ij}$
are exactly the same as in \S\ref{sect.Poin}.
In particular, the equalities (\ref{c0c2})--(\ref{d[1]-d[2]}) with
${\mathfrak s}(x)=\sqrt{-1}\,\sin(x/2)$, and (\ref{last.res}) with $\varepsilon=-1$
hold for spherical frameworks.

\begin{lemma}[cf.~Lemma \ref{lem.propor}]\label{lem.proporS}
If $u\ne 0$, $\bf p$ is $\mathbb P^2$-non-overlapping, and $F_1=\lambda F_2$,
then $\bf p$ satisfies (\SDixI).
\end{lemma}

\begin{proof}
The hyperbolic proof of the case $\lambda(1-\lambda)\ne0$ of
Lemma~\ref{lem.propor} (see \S\ref{sect.Poin})
goes without any change in the spherical setting
due to the assumption $u\ne 0$ (see Remark~\ref{diff.H.S}).

The proof of Lemma~\ref{lem.propor} in the case $\lambda(1-\lambda)=0$
does not extend immediately for the sphere because $u_{ij}$ may be negative,
but one can apply the arguments as in the case $\lambda(1-\lambda)\ne0$.
Namely, a computation of the Gr\"obner bases shows that the ideals
$$
  \Big(e_{02},\,e_{20},\,e_{03},\,e_{30},\,e_{04},\,\lambda,\,
  1+\sum_j(u_{0j}-u_{1j})v_j,\,
  1+\sum_j(u_{0j}+u_{1j})w_j,\,1+u_{00}z\Big),
$$
\vskip-15pt
$$
  \Big(e_{11},\,e_{02},\,e_{20},\,e_{03},\,e_{30},\,e_{04},\,
  1-\lambda,\,
  1+\sum_j(u_{1j}-u_{2j})v_j,\,
  1+\sum_j(u_{1j}+u_{2j})w_j,\,1+u_{00}z\Big)
$$
of the ring $\Lambda[v_0,v_1,v_2,w_0,w_1,w_2,z]$ contain $1$.
This means (cf.~the proof of Lemma~\ref{lem.propor}) that the condition
$\lambda=k$ (where $k=0,1$) combined with $u\ne0$ implies that
$p_1$ and either $p_{2k}$ or its antipode are equidistant from each $q_j$.
Hence $\bf p$ satisfies (\SDixI)
by Lemma~\ref{collin}.
\end{proof}


\subsection{Reducibility conditions for $\hat F_i$}

Let the notation be as \S\ref{sect.reduc} (adapted for the spherical case).
We assume that $M$ contains a flexible $\mathbb P^2$-non-overlapping
$(3,3)$-framework $\bf p$.
Recall that $A_j^\pm=r_{1j}\pm r_{0j}$, $j=0,1,2$,
and $\mathfrak s(x)=\sqrt{-1}\,\sin(x)$.
As in \S\ref{sect.reduc}, we simplify the notation setting
$T=T_1$, $S=S_1$, $R=R_1$, $a_j=r_{1j}$.
Without loss of generality we may assume that
\begin{equation}\label{2r<pi}
2r_{ij}\le\pi\qquad\text{ when $i=0$ or $j=0$}
\end{equation}
(this condition can be always achieved
replacing some joints by their antipodes).
Lemma~\ref{nuiz} combined with (\ref{2r<pi}) implies that for $j,k=0,1,2$, $j\ne k$,
we have:
\begin{align}
   {\mathfrak s}(A_j^+ - A_k^-)\ne 0,\qquad
   {\mathfrak s}(A_j^+ &+ A_0^-)\ne 0,\qquad
   {\mathfrak s}(A_j^- + A_0^+)\ne 0,             \label{AA.1}
\\
  {\mathfrak s}(A_j^- + A_0^-)=0\quad&\Rightarrow\quad A_j^- + A_0^-=0,  \label{AA.2}
\\
  {\mathfrak s}(A_j^\pm-A_k^\pm)=0\quad&\Rightarrow\quad A_j^\pm-A_k^\pm=0,\label{AA.3}
\\
  {\mathfrak s}(A_j^+ + A_k^\pm)=0\quad&\Rightarrow\quad
                      \text{$A_j^+ + A_k^\pm = 2\pi\;$ and $\;jk\ne 0$}.   \label{AA.4}
\end{align}

Abusing the language, we define parallelogrammatic cycles and (anti)-parallelograms
as in \S\ref{sect.Loba}.
Say that a 4-cycle is a {\it $\mathbb P^2$-(anti)-parallelogram}
(resp. {\it $\mathbb P^2$-deltoid\/}) if it becomes an (anti)-parallelogram
(resp. deltoid) after applying the antipodal involution to some vertices.

\begin{lemma}[cf.~Lemma~\ref{lem.f[i,j]}]\label{lem.f[i,j]S}
The polynomial $\hat f_{1j}$, $j=1,2$, is reducible over $\mathbb C$ if and only if
the 4-cycle $p_0q_0p_1q_j$ either is $\mathbb P^2$-parallelogrammatic or it is
a $\mathbb P^2$-deltoid.
\end{lemma}

\begin{proof}
Let $\hat f_{ij}$ be reducible. Suppose that $\hat f_{ij}$ has a
factor $\hat f_0$ of degree zero in $t_j$.
Write $\hat f_{1j}=c_2 t_j^2+c_1 s_j t_j + c_0 s_j^2$.
As in \S\ref{sect.Poin}, we have
$\text{Res}_{\,T}(c_0,c_2)\doteq(U_1^2-u^2)S^4$, which implies
$U_1=u$ because $U_1,u\ge 0$ by (\ref{2r<pi}). Then
$c_0$ and $c_2$ are as in (\ref{c0c2}). Thus either $\hat f_0=S-T$, or
$u=0$ and $\hat f_0=S+T$. If $\hat f_0=S-T$, we conclude
(as in \S\ref{sect.reduc} and \S\ref{sect.Poin}) that $U_1=u$ and $u_1=u_{11}$,
which corresponds to a deltoid. If $u=0$ and $\hat f_0=S+T$,
then $c_1$ must vanish identically in $T$
after the substitutions $U_1=u=0$, $S=-T$. Performing this substitution, we obtain
$c_1=4(u_j+u_{1j})T^2$. Hence $u_{1j}=-u_j$ and $U_1=u=0$. If we replace $p_1$ by its
antipode, we change the sign of $u_{1j}$ and again obtain a deltoid.

Suppose now that $\hat f_{ij}$ does not have any
factor of degree zero in $t_j$. As in the proof of Lemma~\ref{lem.f[i,j]}, we have
to consider the following two cases.

\smallskip
{\it Case 1.} $d_j^+ \equiv d_j^-$. This is impossible because
$d_j^+ - d_j^- = 4T\sin r_j\sin a_j$
(by (\ref{d[j]}) with ${\mathfrak s}(x)=\sqrt{-1}\,\sin(x/2)$)
and $r_j,a_j\in{]}0,\pi{[}$ since $\bf p$ is $\mathbb P^2$-non-overlapping.

\smallskip
{\it Case 2.} $\Delta_j^+=\Delta_j^-=0$.
By (\ref{De[j]}) combined with (\ref{AA.1})--(\ref{AA.4}),
we then have
$$
  (A_j^+ + A_0^+ - 2\pi)(A_j^+ - A_0^+)=
  (A_j^- + A_0^-)(A_j^- - A_0^-)=0.
$$
This is equivalent to four systems of linear equations. Two of them are (\ref{eq.f[i,j]}).
The other two are equivalent to $a_j+R=r_j+r=\pi$ and $a_j+r=r_j+R=\pi$.
Applying the antipodal involution to $q_j$ (for the former system) or to $p_1$
(for the latter one), we obtain a deltoid or a parallelogrammatic cycle respectively.
\end{proof}

\begin{lemma}[cf.~Lemma~\ref{lem.F[1]}]\label{lem.F[1]S}
If $\hat f_{11}$ and $\hat f_{12}$ are irreducible and
$\hat F_1$ is a non-zero reducible polynomial which is not a power
of an irreducible polynomial, then the 4-cycle
$p_0q_1p_1q_2$ either is $\mathbb P^2$-parallelogrammatic,
or it is a $\mathbb P^2$-deltoid with axis $p_0p_1$.
\end{lemma}

\begin{proof}
The arguments are as in the proof of Lemma~\ref{lem.F[1]}
but more cases are to be considered.

\smallskip
{\it Case 1.} $d_1^+\equiv d_2^+$ and $d_1^-\equiv d_2^-$.
By (\ref{d[1]-d[2]}) combined with (\ref{AA.2})--(\ref{AA.4}), we have
\begin{equation}\label{eqsAA}
  (A_2^+ + A_1^+ - 2\pi)(A_2^+ - A_1^+)=
  (A_2^- + A_1^-)       (A_2^- - A_1^-)=0.
\end{equation}
This gives us four systems of equations:
(\ref{eqsA}) and two more systems that are equivalent to
$a_1+a_2=r_1+r_2=\pi$ or $a_1+r_2=r_1+a_2=\pi$.
Applying the antipodal involution to $q_1$ or $p_1$,
we obtain a deltoid or a parallelogrammatic cycle respectively.

\smallskip
{\it Case 2.} $d_1^-\equiv d_2^-$ and $\Delta_1^+ = \Delta_2^+ = 0$.
Due to (\ref{De[j]}), (\ref{AA.1})--(\ref{AA.4}),
the second condition yields
$$
  (A_1^+ + A_0^+ - 2\pi)(A_1^+ - A_0^+)=
  (A_2^+ + A_0^+ - 2\pi)(A_2^+ - A_0^+)=0.
$$
Elimination of $A_0^+$ yields $(A_2^+ + A_1^+ - 2\pi)(A_2^+ - A_1^+)=0$.
With $d_1^-\equiv d_2^-$, this yields (\ref{eqsAA}).

\smallskip
{\it Case 3.} $d_1^+\equiv d_2^+$ and $\Delta_1^- = \Delta_2^- = 0$.
Due to (\ref{De[j]}), (\ref{AA.1})--(\ref{AA.4}),
the second condition yields (\ref{eqsA.3}).
Eliminating $A_0^-$ from it, we obtain $(A_2^- + A_1^-)(A_2^- - A_1^-)=0$.
Combining it with $d_1^+\equiv d_2^+$, we again obtain (\ref{eqsAA}).

\smallskip
{\it Case 4.} $d_1^+\equiv d_2^-$ and $d_1^-\equiv d_2^+$.
By (\ref{d[1]-d[2]}) combined with (\ref{AA.1}) and (\ref{AA.4}),
we obtain $A_1^+ + A_2^- = A_1^- + A_2^+ = 2\pi$, whence $r_1=r_2$
and $a_1+a_2=2\pi$.
A contradiction.

\smallskip
{\it Case 5.} $d_1^+\equiv d_2^-$ and $\Delta_1^- = \Delta_2^+ = 0$
(the same arguments for 1 and 2 swapped). By (\ref{d[1]-d[2]}) and (\ref{AA.1}),
the condition $d_1^+\equiv d_2^-$ implies ${A_1^+} + {A_2^-} = 2\pi$.
By (\ref{De[j]}) and (\ref{AA.1}),
the conditions $\Delta_1^-=0$ and $\Delta_2^+=0$ imply
$\pm{A_0^-} - {A_1^-} = 0$ and $({A_0^+} - \pi) \pm ({A_2^+} - \pi) = 0$
respectively.
Summing up the three equations divided by 2, we obtain
$r_1 + c = 2\pi - a_2$ or $r_1 + c = \pi + r_2$ where $c$ is $r$ or $R$.
This fact contradicts (\ref{2r<pi}).
\end{proof}

We summarize Lemmas \ref{lem.f[i,j]S}, \ref{lem.F[1]S}, and \ref{lem.F[1].square}
as follows.

\begin{lemma}[cf.~Lemma~\ref{main.lemma}]\label{main.lemmaS}
Suppose that $u\ne 0$ and $\bf p$ does not satisfy (\SDixI).
Then the $(2,3)$-framework $(p_0,p_1;\,q_0,q_1,q_2)$ contains either
a $\mathbb P^2$-parallelogrammatic cycle, or a fastened
$\mathbb P^2$-deltoid, or a not fastened $\mathbb P^2$-deltoid with
axis $p_0p_1$.
\end{lemma}


\subsection{Completing the proof of Theorem~\ref{tS}}\label{sect.EOPS}

Let ${\bf p}=(p_0,p_1,p_2;\,q_0,q_1,q_2)$
be a flexible $\mathbb P^2$-non-overlapping spherical $(3,3)$-framework
which does not satisfy (\SDixI). Let us show that $\bf p$
satisfies (\PDixII) or (\DA). In this subsection we do not identify $S^2$ with
$\mathbb C\cup\{\infty\}$, thus $p_i$ and $q_j$ are just points in
$S^2\subset\mathbb R^3$, and $-p_i$ is the antipode of $p_i$.
As above, we set $u_{ij} = \langle p_i,q_j\rangle = \cos d_S(p_i,p_j)$.

\begin{lemma}\label{0000}
(a). $\bf p$ cannot contain a rhombus with the side length $\pi/2$.

(b). {\rm(Follows from Lemma~\ref{collin})}
    $(u_{0j},u_{1j},u_{2j})\ne(0,0,0)$ for any $j=0,1,2$.
\end{lemma}

\begin{lemma}[cf.~Lemma \ref{lem.deltoid}]\label{lem.deltoidS}
If $\bf p$ contains a $\mathbb P^2$-deltoid which is not a $\mathbb P^2$-rhombus,
then $\bf p$ satisfies (\DA).
\end{lemma}

\begin{proof}
Suppose that $\bf p$ contains a $\mathbb P^2$-deltoid
$\Delta$ which is not a $\mathbb P^2$-rhombus. Without loss of generality
we may assume that $\Delta$ is a deltoid.
Renumber the joints so that
$\Delta=p_0q_1p_1q_2$ and the axes of $\Delta$ is $q_1q_2$, i.e.,
$u_{01}=u_{11}\ne u_{02}=u_{12}$
(see Fig.~\ref{fig.lem}, on the left).

If $u_{00}=u_{10}$, then $p_0$ and $p_1$ are equidistant from each $q_j$
and $\bf p$ satisfies (\SDixI) by Lemma~\ref{collin}.
Hence one of $u_{00},u_{10}$ is non-zero. Up to exchange of $p_0$ and $p_1$,
we may assume that $u_{00}\ne0$.
Hence, by Lemma~\ref{main.lemmaS}, there exist $4$-cycles $\Delta'$ and $\Delta^*$
such that: $\Delta'$ is contained in $(p_0,p_1;\,q_0,q_1,q_2)$,
$\Delta^*$ is obtained from $\Delta'$ by the antipodal involution applied to some joints,
and $\Delta^*$ realizes one of the cases considered below. In each case we treat
only the subcases not covered in the proof of Lemma~\ref{lem.deltoid}.
We consider the subcases up to swapping $p_0\leftrightarrow p_1$
and  $q_1\leftrightarrow q_2$.
If $p_i,q_j$ are vertices of $\Delta'$, we denote the corresponding
vertices of $\Delta^*$ by $p_i^*,q_j^*$ and we set
$u_{ij}^*=\langle p_i^*, q_j^*\rangle$.

\smallskip

{\it Case 1. $\Delta^*$ is a parallelogrammatic cycle.}

\smallskip
{\it Subcase 1a.} $\Delta'=\Delta$.
Let ${\bf u}^*=(u_{01}^*,u_{02}^*,u_{11}^*,u_{12}^*)$.
\begin{itemize}
\item
If ${\bf u}^*=(-u_{01},-u_{02},u_{11},u_{12})$,
then $-u_{01}=u_{12}=u_{02}=-u_{11}$, hence $(p_0,-q_1,p_1,q_2)$ is a rhombus,
thus $\Delta$ is a $\mathbb P^2$-rhombus.
\item
If ${\bf u}^*=(-u_{01},u_{02},-u_{11},u_{12})$,
then $\Delta^*$ is a rhombus.
\item
If ${\bf u}^*=(-u_{01},u_{02},u_{11},-u_{12})$,
then $\Delta$ is a rhombus.
\end{itemize}

\smallskip
{\it Subcase 1b.} $\Delta'=p_0q_0p_1q_1$ and
$(u_{00}^*,u_{01}^*,u_{10}^*,u_{11}^*)=
(u_{00},\varepsilon u_{01},\varepsilon\delta u_{10},\delta u_{11})$,
$\delta,\varepsilon=\pm 1$.
Then $u_{00}=\delta u_{11}=\delta u_{01}=u_{10}$,
hence $p_0$ and $p_1$ are equidistant from
each $q_j$, which contradicts Lemma~\ref{collin} because of our assumption that
$\bf p$ does not satisfy (\SDixI).

\smallskip
{\it Case 2. $\Delta^*$ is a fastened deltoid with axis $p_0p_1$.}
We may assume that $\Delta'=p_0q_0p_1q_1$
and
$(u_{00}^*,u_{01}^*,u_{10}^*,u_{11}^*)=
(u_{00},\varepsilon u_{01},\varepsilon\delta u_{10},\delta u_{11})$,
$\delta,\varepsilon=\pm 1$.
Then $u_{00}=\varepsilon u_{01}=\varepsilon u_{11}=u_{10}$
and we conclude as in Subcase 1b.

\smallskip
{\it Case 3. $\Delta^*$ is a fastened deltoid with axis $q_0q_j$}
(the most interesting case).
We may assume that $\Delta'=p_0q_0p_1q_2$.
Let ${\bf u}^*=(u_{00}^*,u_{02}^*,u_{10}^*,u_{12}^*)$.
If ${\bf u}^*=(-u_{00},u_{02},-u_{10},u_{12})$, then $u_{00}=u_{10}$
and we conclude as in Subcase 1b.
Otherwise we may assume that
${\bf u}^*=(-u_{00},\varepsilon u_{02},u_{10},-\varepsilon u_{12})$, $\varepsilon=\pm 1$.
Then $u_{02}=u_{12}$ (since $\Delta$ is a deltoid with axis $q_1q_2$) and
$u_{02}=-u_{12}$ (since $\Delta^*$ is a deltoid with axis $q_0q_2$).
Hence $u_{02}=u_{12}=0$. We also have $u_{00}=-u_{10}$. Set
$u=u_{10}=-u_{00}$ and $v=u_{01}=u_{11}$. We have $uv\ne 0$ by Lemma~\ref{0000}(a).

\begin{figure}[ht]
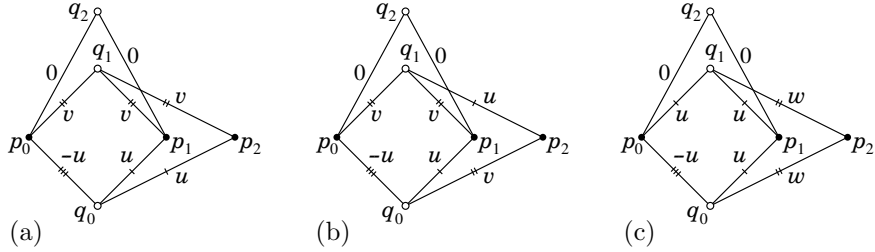

\centerline{
(a)\hskip-14pt\includegraphics[width=35 mm]{lem-cda1.eps}\hbox to 5mm{}
(b)\hskip-14pt\includegraphics[width=35 mm]{lem-cda2.eps}\hbox to 5mm{}
(c)\hskip-14pt\includegraphics[width=35 mm]{lem-cda3.eps}
}
\caption{The $u_{ij}$'s in Case 3 in the proof of Lemma \ref{lem.deltoidS}.}
\label{fig.lem.deltoidS}
\end{figure}

Consider the $(3,2)$-framework $(p_0,p_1,p_2;\,q_0,q_2)$.
By Lemma~\ref{main.lemmaS} it contains either a $\mathbb P^2$-parallelogrammatic cycle
or a $\mathbb P^2$-deltoid. Then one can check that, up to renumbering and
antipodal involutions, the $u_{ij}$'s are as in Figure~\ref{fig.lem.deltoidS}.
The $(2,3)$-framework $(p_0,p_2;\,q_0,q_1,q_2)$ also 
contains a $\mathbb P^2$-parallelogrammatic cycle
or a $\mathbb P^2$-deltoid. Since $u_{22}\ne 0$ by Lemma~\ref{0000}(b),
this is possible only when $w=0$ in Figure~\ref{fig.lem.deltoidS}(c), which means
that {\bf p} satisfies (\DA).

\smallskip
{\it Case 4. $\Delta'=\Delta$ and $\Delta^*$ is a deltoid with axis $p_0p_1$.}
Let ${\bf u}^*=(u_{01}^*,u_{02}^*,u_{11}^*,u_{12}^*)$.
\begin{itemize}
\item
If ${\bf u}^*=(-u_{01},-u_{02},u_{11},u_{12})$,
then $\Delta$ is a rhombus.
\item
If ${\bf u}^*=(-u_{01},u_{02},-u_{11},u_{12})$,
then $\Delta^*$ is a rhombus.
\item
If ${\bf u}^*=(-u_{01},u_{02},u_{11},-u_{12})$,
then $-u_{01}=u_{02}=u_{12}=-u_{11}$, hence $(p_0,-q_1,p_1,q_2)$ is a rhombus,
thus $\Delta$ is a $\mathbb P^2$-rhombus.
\end{itemize}
\end{proof}

\begin{lemma}
\label{lem.paral.sph}
Let $\Pi$ be a $\mathbb P^2$-parallelogrammatic cycle and $\Pi^*$ be
obtained from $\Pi$ by applying the antipodal involution to one of its vertices.
Then either $\Pi$ or $\Pi^*$ is parallelogrammatic.
\end{lemma}

\begin{lemma}[cf.~Lemma~\ref{lem.paral}]
$\bf p$ cannot contain two distinct $\mathbb P^2$-parallelogrammatic
cycles with three common vertices
\end{lemma}

\begin{proof} Combine the proof of Lemma~\ref{lem.paral} with Lemma~\ref{lem.paral.sph}.
\end{proof}

The rest of \S\ref{sect.EOP} easily extends to
the spherical case using Lemma~\ref{lem.paral.sph}, but the following
additional argument is needed at the final step.

Suppose that $\bf p$ does not satisfy (\PDixII).
Then $u+u_1+u_2+u_{12}\ne 0$ because otherwise $(p_0,-p_1,p_2;\,q_0,-q_1,q_2)$
would satisfy (\SDixII).
Recall that the spherical version of (\ref{res.bcd})
is (\ref{last.res}) with $\varepsilon=-1$. This product vanishes for each choice
of the sign ``$\pm$'' only if $l=l_2=1$, i.e.,
only if $u=u_2=0$, but this condition contradicts Lemma~\ref{0000}(a).


\end{document}